\documentclass{amsart}
\usepackage{amsmath, amsfonts, amsthm, framed}
\usepackage{xspace}
\usepackage{newunicodechar}
\usepackage{fullpage}
\usepackage{enumitem}
\usepackage{quiver}
\usepackage{fontawesome}
\usepackage[numbers]{natbib}
\usepackage{soul}
\usepackage{cancel}
\usepackage{xspace}
\usepackage{eucal}
\usepackage[colorlinks=true]{hyperref}

\definecolor{darkergreen}{rgb}{0.0, 0.5, 0.0}
\definecolor{Blue}{RGB}{0,0,148}

\usepackage{listings}

\definecolor{keywordcolor}{rgb}{0.7, 0.1, 0.1}   
\definecolor{commentcolor}{rgb}{0.4, 0.4, 0.4}   
\definecolor{symbolcolor}{rgb}{0, 0, 0.8}    
\definecolor{tacticcolor}{rgb}{0, 0, 0.8}    
\definecolor{sortcolor}{rgb}{0.1, 0.5, 0.1}      
\newcommand*{\lean}[1]{\lstinline{#1}\xspace} 
\theoremstyle{plain}
\newtheorem{theorem}{Theorem}[section]
\newtheorem*{theorem_no}{Theorem}
\newtheorem{lemma}[theorem]{Lemma}
\newtheorem{proposition}[theorem]{Proposition}
\theoremstyle{definition}
\newtheorem{definition}[theorem]{Definition}
\newtheorem{remark}[theorem]{Remark}

\newcommand*{\N}{\mathbb{N}}

\newcommand*{\C}{\mathcal{C}}
\newcommand*{\D}{\mathcal{D}}
\newcommand*{\T}{\mathcal{T}}
\newcommand*{\Hom}{\operatorname{Hom}}
\newcommand*{\op}{\operatorname{op}}
\newcommand*{\dom}{\operatorname{dom}}
\newcommand*{\image}{\operatorname{im}}
\newcommand*{\Set}{\mathsf{Set}}
\newcommand*{\Sh}{\operatorname{Sh}}
\newcommand*{\Top}{\mathsf{Top}}
\newcommand*{\CompHaus}{\mathsf{CompHaus}}
\newcommand*{\Profinite}{\mathsf{Profinite}}
\newcommand*{\Stonean}{\mathsf{Stonean}}
\makeatletter
\newenvironment*{g@n@riclist}[1]{%
    \begin{enumerate}[label=#1]
    }{%
    \end{enumerate}
    }
\newenvironment{list_conditions}{
    \begin{g@n@riclist}{\arabic*\textup{)}}
    }{%
    \end{g@n@riclist}
    }
\newenvironment{list_named_conditions}[1]{
    \begin{g@n@riclist}{#1-\arabic*\textup{)}}
    }{%
    \end{g@n@riclist}
    }
\newenvironment{list_tfae}{
    \begin{g@n@riclist}{\textup{(}\roman*\textup{)}}
    }{%
    \end{g@n@riclist}
    }
\newenvironment{list_cases}{
\begin{g@n@riclist}{\alph*\textup{)}}
}{%
\end{g@n@riclist}
}
\makeatother
\newcommand*{\Cech}{Čech\xspace}
\newcommand*{\mathlib}{\textsc{mathlib}\xspace} 
\newcommand*{\Lean}{\textsc{Lean}\xspace} 

\title{Categorical foundations of formalized condensed mathematics}
\author{Dagur Asgeirsson}
\address[D.~A.]{University of Copenhagen}
\email{dagur@math.ku.dk}
\author{Riccardo Brasca}
\address[R.~B.]{Institut de Mathématiques de Jussieu-Paris Rive Gauche, Université Paris-Cité}
\email{riccardo.brasca@imj-prg.fr}
\author{Nikolas Kuhn}
\address[N.~K.]{University of Oslo}
\email{ntkuhn@posteo.net}
\author{Filippo Alberto Edoardo Nuccio Mortarino Majno di Capriglio}
\address[F.~N.]{ Université Jean Monnet, CNRS, Ecole Centrale de Lyon, INSA Lyon, Universite Claude Bernard Lyon~1, ICJ~UMR5208, 42023 Saint-Etienne, France}
\email{filippo.nuccio@univ-st-etienne.fr}
\author{Adam Topaz}
\address[A.~T.]{University of Alberta}
\email{topaz@ualberta.ca}
\begin{document}

\begin{abstract}
    Condensed mathematics, developed by Clausen and Scholze over the last few years, proposes a generalization of topology with better categorical properties. 
    It replaces the concept of a topological space by that of a condensed set, which can be defined as a sheaf for the coherent topology on a certain category of compact Hausdorff spaces. 
    In this case, the sheaf condition has a fairly simple explicit description, which arises from studying the relationship between the coherent, regular and extensive topologies.
    In this paper, we establish this relationship under minimal assumptions on the category, going beyond the case of compact Hausdorff spaces. 
    Along the way, we also provide a characterization of sheaves and covering sieves for these categories. 
    All results in this paper have been fully formalized in the \Lean proof assistant.
\end{abstract}

\maketitle

\section{Introduction}
The main goal of condensed mathematics (see e.g.~\cite{analytic, condensed, complex}) is to provide a better framework to study the interplay between algebra and geometry. 
To do this, one has to generalize the notion of a topological space to obtain better categorical properties; the category of condensed sets achieves this remarkably well.
A condensed set is defined as a sheaf for the so-called \emph{coherent topology} on the category of compact Hausdorff spaces. 
The category of condensed sets contains a very large class of topological spaces as a full subcategory. In addition, it almost forms a topos\footnote{There are some set-theoretic issues that prevent it from satisfying all the axioms of a topos; these can be resolved in various ways and, for all practical purposes, the category of condensed sets can be regarded as a topos.}, and the category of condensed abelian groups is a particularly well-behaved abelian category.

The formalization of the theory of condensed sets started with the \emph{Liquid Tensor Experiment}, see~\cite{LTE, LTE_challenge}. 
In that work the authors formalized the definition and various properties of the category of condensed abelian groups and of liquid vector spaces, including the main result~\cite[Theorem 9.1]{analytic}, using the \Lean proof assistant. 
In~\S\ref{sec:mathlib} we will offer a brief outline both of \Lean and of its main mathematical library \mathlib.

Even if the achievements of the \emph{Liquid Tensor Experiment} are spectacular, most of the work is not suitable to be integrated into a large mathematical library like \mathlib. 
Indeed, a lot of results in the Liquid Tensor Experiment were stated and proven in an \emph{ad-hoc} way and are not applicable in other contexts. 
This approach contradicts many of the design decisions prevalent throughout \mathlib, which we briefly discuss in~\S\ref{subsec:mathlib_generality}.

The main goal of our work is to formalize the foundations of the theory of condensed sets in an organic way, being as general as possible in all the various prerequisites. 
Indeed, the present work has already been incorporated in the \mathlib library. 
Besides correctness, which is checked by \Lean, this ensures that the results are stated in a way that is compatible with the rest of the library and that they can be used by others.

The goal of this paper is to prove, in the most general setting, results relating the coherent, regular and extensive topologies on a category, as well as characterizations of their sheaves.
While the results we discuss in this paper are known to some experts as part of the folklore, we provide both a detailed exposition, while simultaneously minimizing various assumptions.
The more general approach we take in this paper was motivated primarily by the formalization of these results.

Throughout the text, we use the symbol \faExternalLink\xspace for external links. 
Almost every mathematical statement and definition will be accompanied by such a link directly to the source code for the corresponding statement in \mathlib. The only exceptions are results that we use in the informal proof but not in the formal one. In particular, all relevant results are completely formalized in \mathlib. In order for the links to stay usable, they are all to a fixed commit to the master branch (the most recent one at the time of writing).

Here is a brief outline of the paper.  In~\S\ref{sec:mathlib} we give a brief overview of the \Lean proof assistant and its mathematical library \mathlib, explaining the general philosophy behind the library and the main design decisions that have been taken, focusing on the aspects that are most relevant to the present work. In~\S\ref{sec:preliminaries}, we review the theory of sheaves for Grothendieck topologies as it is formalized in \mathlib: this section is standard, but we think it is a good idea to fix the notation and the terminology, as the literature is not always consistent. In~\S\ref{sec:effective}, we introduce the notions of strict, regular and effective epimorphism. We prove in Proposition~\ref{prop:TocCatEffectiveEpi} that the effective epimorphisms in the category~$\C$ of topological spaces are the quotient maps and Proposition~\ref{prop:CompHausEffectiveEpi} characterizes effective epimorphisms in $\C$ as the continuous surjections. Strict, regular, and effective epimorphisms are then used in~\S\ref{sec:topologies} to define the regular (\emph{resp.} extensive, coherent) topology on a category satisfying the technical condition of being preregular (\emph{resp.} finitary extensive, precoherent). We prove in Proposition~\ref{prop:precoherent_of_preregular_extensive} that a preregular and finitary extensive category is precoherent and in Proposition~\ref{prop:extensive_regular_generate_coherent} that the coherent topology is generated by the union of the regular and extensive topologies. In~\S\ref{sec:sheaves}, we study sheaves on these three topologies: first of all we prove in Propositions~\ref{prop:regular-subcanonical}, \ref{prop:extensive-subcanonical}, and~\ref{prop:coherent-subcanonical} that the three topologies are subcanonical. We then give in Propositions~\ref{prop:regularTopology.isSheaf_of_projective}, \ref{prop:isSheaf_iff_preservesFiniteProducts}, \ref{prop:regular_extensive_sheaf}, and~\ref{prop:regular_extensive_sheaf_projective} various conditions for a presheaf to be a sheaf (characterizing sheaves in terms of the preservation of finite products and equalizers). We then give in Proposition~\ref{prop:sheafEquiv} a condition for a functor\footnote{In this work we follow the convention that all functors are, by definition, covariant; we refer to contravariant functors as \emph{presheaves}.} to induce an equivalence between the categories of sheaves for certain topologies. In~\S\ref{sec:back_to_condensed} we apply our general categorical framework to the theory of condensed sets, proving our main theorems, that we now summarize.

Consider the following three categories, each containing the next as a full subcategory, and whose morphisms are continuous maps:
\begin{itemize}
    \item $\CompHaus$: the category of compact Hausdorff spaces~\href{https://github.com/leanprover-community/mathlib4/blob/743032e7ead097fb3e8ae5cd02d29cdd8899161c/Mathlib/Topology/Category/CompHaus/Basic.lean#L41-L49}{\faExternalLink}.
    \item $\Profinite$: the category of \emph{profinite} spaces, that we define, following \mathlib, as totally disconnected compact Hausdorff spaces. This category is equivalent to the pro-category of the category of finite sets (this last statement has not yet been formalized; see~\cite[Section 6]{dagur-itp-2024} for a more detailed discussion of the state of the category $\Profinite$ in \mathlib)~\href{https://github.com/leanprover-community/mathlib4/blob/743032e7ead097fb3e8ae5cd02d29cdd8899161c/Mathlib/Topology/Category/Profinite/Basic.lean#L50-L55}{\faExternalLink}.
    \item $\Stonean$: the category of \emph{Stonean} spaces, whose objects are extremally disconnected compact Hausdorff spaces~\href{https://github.com/leanprover-community/mathlib4/blob/743032e7ead097fb3e8ae5cd02d29cdd8899161c/Mathlib/Topology/Category/Stonean/Basic.lean#L43-L48}{\faExternalLink}. The condition of being extremally disconnected means that the closure of every open set is open. These spaces are precisely the projective objects in $\CompHaus$ (see~\cite[Theorem~2.5]{gleason} and \href{https://github.com/leanprover-community/mathlib4/blob/743032e7ead097fb3e8ae5cd02d29cdd8899161c/Mathlib/Topology/Category/Stonean/Basic.lean#L310-L319}{\faExternalLink}). It is easy to see that Stonean spaces are totally disconnected, so we have a fully faithful inclusion $\Stonean\subseteq\Profinite$ \href{https://github.com/leanprover-community/mathlib4/blob/743032e7ead097fb3e8ae5cd02d29cdd8899161c/Mathlib/Topology/Category/Stonean/Basic.lean#L131-L144}{\faExternalLink}.
\end{itemize}

Let $\C$ be any of these categories.
We prove in Proposition~\ref{prop:precoherent_of_preregular_extensive} and Proposition~\ref{prop:preregular_finitary_extensive} that the categories $\C$ fit into the general framework we describe in this paper.
As a consequence, we recover the following two key results (stated here as Theorem~\ref{thm:sheaves} and~\ref{thm:condensed set equivalence}) which have appeared early on in the theory of condensed mathematics~\cite[Definition~1.2 and Proposition~2.7]{condensed}.
\begin{theorem_no}
    We have the following characterizations of sheaves on $\C$.    
    \begin{itemize}
        \item  When $\C$ is $\CompHaus$ or $\Profinite$, a presheaf $X \colon \C^{\op} \to \Set$ is a sheaf for the coherent topology on $\C$ if and only if it satisfies the following two conditions:
        \begin{list_conditions}
            \item $X$ preserves finite products: in other words, for every finite family $(T_i)$ of objects of $\C$, the natural map 
            \[
                X\Bigl(\coprod_i T_i\Bigr) \longrightarrow \prod_i X(T_i)    
            \]
            is a bijection.
            \item For every surjection $\pi \colon S \to T$ in $\C$, the diagram 
            \[\begin{tikzcd}
                {X(T)} & {X(S)} & {X(S \times_T S)}
                \arrow["{X(\pi)}", from=1-1, to=1-2]
                \arrow[shift left, from=1-2, to=1-3]
                \arrow[shift right, from=1-2, to=1-3]
            \end{tikzcd}\]
            is an equalizer \textup{(}the two parallel morphisms being induced by the projections in the pullback\textup{)}. 
        \end{list_conditions}
        \item A presheaf $X \colon \Stonean^{\op} \to \Set$ is a sheaf for the coherent topology on $\Stonean$ if and only if it preserves finite products: in other words, for every finite family $(T_i)$ of object of $\C$, the natural map 
        \[
            X\Bigl(\coprod_i T_i\Bigr) \longrightarrow \prod_i X(T_i)    
        \]
        is a bijection.
    \end{itemize}
\end{theorem_no}
\begin{theorem_no}
    The inclusion functors $\Profinite \to \CompHaus$ and $\Stonean \to \CompHaus$ induce equivalences of categories between the categories of sheaves for the coherent topology on $\CompHaus$, $\Profinite$, and $\Stonean$.
\end{theorem_no}
Recall that a condensed set is defined as a sheaf for the coherent topology on $\CompHaus$. Thanks to the second theorem, the category of condensed sets is equivalent to the category of sheaves for the coherent topology on $\Profinite$ or $\Stonean$. 

In fact these theorems hold for very general target categories other than that of sets, they certainly hold for the category of modules over a ring, for example. Regarding condensed objects simply as product-preserving presheaves on $\Stonean$ allows us to perform many constructions ``objectwise'' on $\Stonean$. For example, limits and filtered colimits of condensed sets are given objectwise on $\Stonean$; in the setting of condensed abelian groups or modules, the situation is even better --- \emph{all} colimits are computed objectwise on $\Stonean$. Furthermore, epimorphisms of condensed objects in a sufficiently nice concrete category are simply those morphisms $X \to Y$ which satisfy the property that the induced map $X(S) \to Y(S)$ is surjective for every object~$S$ of $\Stonean$. These two facts are essential in proving that condensed abelian groups form an abelian category which satisfies all the same of Grothendieck's AB axioms as the category of abelian groups. This result has not yet made it into \mathlib, but is well within reach.

\section{Mathlib}\label{sec:mathlib}
The results we describe in this paper have all been formalized using the \Lean interactive theorem prover, and incorporated into its open-source formalized mathematical library \mathlib~\cite{mathlib}.
The \Lean community maintains \mathlib as a large monolith with a number of overarching design decisions, which must be taken into account in all mathematical contributions to it.
This section explains the particulars of \mathlib that played a key motivating role in the presentation results we discuss in this paper.
While we do not provide an introduction to the \Lean theorem prover itself, we refer the reader to~\cite{lean4} for a comprehensive discussion. 

\subsection{Mathematical cohesion} 

One of the key design decisions made in \mathlib is that it strives to be a \emph{cohesive} library.
This point of view manifests concretely in a few ways.
Most notably, it often means that mathematical concepts usually have one ``official'' definition in \mathlib, and various related definitions and lemmas are built around such official definitions (this collection of ancillary results is often referred to as ``the API'') allowing users to work with them effectively.
The importance of this approach cannot be understated when it comes to formalization of advanced mathematics.

\mathlib allows formalizers to efficiently use the constructions from the library, even when their work lies at the intersection of several subjects, which condensed mathematics certainly does. To take a small example, the definition of a condensed set mentions the category of compact Hausdorff spaces, and one frequently has to use both the topological properties of the objects of this category and the more abstract properties of the category itself. The cohesive nature of \mathlib ensures that the interplay between these two aspects of compact Hausdorff spaces runs smoothly. 
This is in contrast with the alternative approach where there are separate libraries for different areas of mathematics, which can potentially be problematic should the same concept appear in two different libraries following different conventions, since results from one library would not be directly compatible with results in the other

\subsection{The ``right'' generality}\label{subsec:mathlib_generality}

A related and equally important design decision in \mathlib is that mathematical contributions should be developed in the ``right'' level of generality.
Although the utility of this approach is clear --- a more general result applies in more contexts --- it is often more convenient for mathematicians to work in the correct level of generality \emph{for their current project}.
However, when making contributions to \mathlib, formalizers are encouraged to keep in mind the cohesive and interconnected nature of the library, since it is often impossible to know how an initial contribution may be used in the future, and in what context.

Nevertheless, it is important to mention that it is usually difficult to find the right level of generality for \mathlib at first.
It often happens that preexisting code in \mathlib is \emph{refactored} to bring it closer to \mathlib prescribed ideals.
In fact, such a refactoring process often occurs in conjunction or in parallel with API development as discussed above.

\subsection{Contribution Process}

Ensuring that the design decisions of \mathlib are maintained requires significant experience with the library.
In practice, this means that contributions must pass a process resembling peer review, whereby ``pull requests'' are opened for potential contributions, which are then reviewed by a team of reviewers and maintainers before being incorporated into the library.

\subsection{Condensed Mathematics}

Having discussed some of the key design decisions of \mathlib, and how these relate to contributions of formalized mathematics within the library, it should come as no surprise that the development of condensed mathematics in \mathlib follows the same lines.
The goal of this paper is to describe the mathematics behind the foundations of condensed mathematics in a way which is suitable for inclusion in \mathlib.
In fact, the general categorical approach we outline in this paper was originally \emph{motivated} by the goal of finding the ``right'' level of generality appropriate for its inclusion in \mathlib.   

\subsection{Size issues}\label{subsec:type_theory}
Condensed mathematics is known to raise subtle set-theoretic issues, see~\cite[Remark 1.3]{condensed}. These can be solved in different ways, one is explained in~\cite[Appendix to Lecture II]{condensed} and another in~\cite[1.2--1.4]{pyknotic}, the latter being closer to the approach used in \mathlib. One advantage of formalizing the theory is to guarantee that all these problems are solved in a precise way. Roughly speaking, the idea is to use Grothendieck's universes. 
These are more or less built into the axiomatic framework of \Lean, which is a version of dependent type theory relying on the calculus of inductive constructions. 
For a more detailed explanation of the foundations of \Lean, we refer the reader to~\cite{mario}.

The basic objects of the theory are \emph{terms} and \emph{types}. 
Every term has a type, and a type can be regarded as a collection of elements, which are the terms of that type. 
In this way, types replace sets in their everyday use in mathematics as ``collections of elements''. 
The notation \lean{a : A} is used to signify that \lean{a} is a term of the type \lean{A}.
To avoid an analogue of Russell's paradox known as Girard's paradox, \Lean uses a \emph{hierarchy of universes} indexed by the natural numbers
\begin{lstlisting}
  Type = Type 0
  Type 0 : Type 1
  Type 1 : Type 2
  ...
\end{lstlisting}
\textsc{Mathlib}'s definition of a category has two universe parameters $u$ and $v$. The definition consists of a ``set of objects'' \lean{(C : Type u)}, and for every pair of objects \lean{X Y : C}, of a ``set of morphisms'' \lean{(X ⟶ Y : Type v)}. Throughout this paper, we will use the word ``set'' informally in this way, letting \Lean take care of making sure that the ``set'' in question has a high enough universe level. 
For a concrete example, see Definition~\ref{def:grothendieck-topology} where we mention the \emph{top sieve} on an object $X$ in a category $\C$. This is supposed to be the ``set of all morphisms in $\C$ with target $X$''. When $\C$ is a large category, this is not a set in the sense of set-theoretic foundations, but as explained above, our use of the word ``set'' is not abusive in this case.

\mathlib's axioms are known to be equivalent to Zermelo--Fraenkel set theory plus the axiom of choice and the existence of $n$ inaccessible cardinals for all $n\in\N$, see~\cite[Corollary~6.8]{mario}. 
In particular, the existence of the hierarchy of universes (and their precise behavior with respect to various constructions) is provable in ZFC using a relatively weak assumption about large cardinals.

\section{Preliminaries}\label{sec:preliminaries}

\subsection{Coverages}
There are various ways to formulate the notion of a site and Grothendieck topology on a category $\C$, which allows us to define the notion of a \emph{sheaf} on $\C$.
In order to fix the terminology, we start this section by recalling some basic definitions and results in this area.
The terminology we describe here matches the terminology used in the corresponding definitions that can be found in \mathlib.

Fix a category $\C$ throughout this section.

\begin{definition}
    \href{https://github.com/leanprover-community/mathlib4/blob/743032e7ead097fb3e8ae5cd02d29cdd8899161c/Mathlib/CategoryTheory/Sites/Sieves.lean#L39-L41}{\faExternalLink}
    Let $X$ be an object of $\C$. 
    A \emph{presieve} $S$ on $X$ is a set of morphisms with target $X$. 
    If $f \in S$ is a morphism, we will use the notation $\dom f$ for the domain of $f$.
\end{definition}

\begin{remark}
  It is sometimes convenient to consider an \emph{indexed} family of morphisms $(f_i : X_i \to X)_{i \in I}$, indexed by some set $I$.
  Of course, any such family yields a presieve $S$ on $X$ which contains only the morphisms $f_i$ for $i \in I$.  
  Conversely, any presieve can be considered as a family indexed by its elements. 

  The notion of an indexed family of morphisms over $X$ is not exactly equivalent to that of a presieve over $X$, as an indexed family may have duplicates while a presieve cannot. 
  However, it is sometimes convenient to use indexed families as opposed to presieves, and we will allow ourselves to freely go back and forth as discussed above.
\end{remark}

\begin{definition}\label{def:compatible_family}
  \href{https://github.com/leanprover-community/mathlib4/blob/743032e7ead097fb3e8ae5cd02d29cdd8899161c/Mathlib/CategoryTheory/Sites/IsSheafFor.lean#L81-L91}{\faExternalLink} 
  \href{https://github.com/leanprover-community/mathlib4/blob/743032e7ead097fb3e8ae5cd02d29cdd8899161c/Mathlib/CategoryTheory/Sites/IsSheafFor.lean#L118-L134}{\faExternalLink}
  \href{https://github.com/leanprover-community/mathlib4/blob/743032e7ead097fb3e8ae5cd02d29cdd8899161c/Mathlib/CategoryTheory/Sites/IsSheafFor.lean#L366-L375}{\faExternalLink}
  Let $F \colon \C^{\op} \to \Set$ be a presheaf on $\C$ and let $S$ be a presieve on an object $X$ of $\C$. A \emph{family of elements} for $S$ is a collection $(x_f)_{f \in S}$ where $x_f \in F(\dom f)$ for all $f \in S$.
  We say that such a family of elements $(x_f)_f$ is \emph{compatible} provided that for all commutative squares in $\C$ of the form
  \[\begin{tikzcd}
      Y & {\dom f} \\
      {\dom f'} & X
      \arrow["{g}", from=1-1, to=1-2]
      \arrow["{f}", from=1-2, to=2-2]
      \arrow["{g'}"', from=1-1, to=2-1]
      \arrow["{f'}"', from=2-1, to=2-2],
  \end{tikzcd}\]
  with $f, f' \in S$, one has $F(g)(x_f) = F(g')(x_{f'})$.
  We say that $x \in F(X)$ is an \emph{amalgamation} for $(x_f)_{f\in S}$ if $F(f)(x) = x_f$ for all $f \in S$.
\end{definition}

\begin{definition}
  \href{https://github.com/leanprover-community/mathlib4/blob/743032e7ead097fb3e8ae5cd02d29cdd8899161c/Mathlib/CategoryTheory/Sites/IsSheafFor.lean#L437-L447}{\faExternalLink}
  We say that a presheaf $F \colon \C^{\op} \to \Set$ \emph{is a sheaf for the presieve} $S$ if for every compatible family of elements for $S$ there exists a unique amalgamation.
\end{definition} 

\begin{remark}\label{rmk:sheaf_with_pb}
If a presieve $S$ on $X$ is constructed out of an indexed family $(f_i : X_i \to X)_{i \in I}$ such that for all $i, j \in I$, the pullback $X_i\times_X X_j$ exists, one can rephrase the sheaf condition for the presieve as saying that the diagram

    \[\begin{tikzcd}
        {F(X)} & {\prod\limits_{i \in I} F(X_i)} & {\prod\limits_{i, j \in I} F(X_i \times_X X_j)}
        \arrow[from=1-1, to=1-2]
        \arrow[shift left, from=1-2, to=1-3]
        \arrow[shift right, from=1-2, to=1-3]
    \end{tikzcd}\]
is an equalizer, where the map on the left is given by the collection $\bigl(F(f_i)\bigr)_{i\in I}$ and the two parallel maps are induced by the projections in the pullbacks. \href{https://github.com/leanprover-community/mathlib4/blob/743032e7ead097fb3e8ae5cd02d29cdd8899161c/Mathlib/CategoryTheory/Sites/IsSheafFor.lean#L792-L795}{\faExternalLink}
\end{remark}

\begin{definition}
  \href{https://github.com/leanprover-community/mathlib4/blob/743032e7ead097fb3e8ae5cd02d29cdd8899161c/Mathlib/CategoryTheory/Sites/Coverage.lean#L136-L151}{\faExternalLink}
  A \emph{coverage} on $\C$ is the datum of a set of presieves on each object $X$ of $\C$, called \emph{covering presieves}, satisfying the following property: 
  For every morphism $f \colon X \to Y$ in $\C$ and every covering presieve $S$ on $Y$, there exists a covering presieve $T$ on $X$ such that for each $g \in T$, the composition $f \circ g$ factors through some morphism $h \in S$.
\end{definition}

\begin{definition}
  \href{https://github.com/leanprover-community/mathlib4/blob/743032e7ead097fb3e8ae5cd02d29cdd8899161c/Mathlib/CategoryTheory/Sites/Sieves.lean#L266-L274}{\faExternalLink}
  A \emph{sieve} $S$ on an object $X$ of $\C$ is a presieve on $X$ which is \emph{downwards closed} in the sense that for each $f \in S$ and every $g$ that is composable with $f$, we have that $f \circ g \in S$. 
  The sieve $\langle R\rangle$ \emph{generated} by a presieve $R$ is the sieve consisting of all morphisms that factor through a morphism of $R$; this is the smallest sieve containing $R$. 
  We also call $\langle R\rangle$ the \emph{sieve associated to $R$}.
\end{definition}

\begin{remark}\label{rem:cocone}
  \href{https://github.com/leanprover-community/mathlib4/blob/743032e7ead097fb3e8ae5cd02d29cdd8899161c/Mathlib/CategoryTheory/Sites/Sieves.lean#L53-L72}{\faExternalLink}
  A sieve $S$ on $X$ can be regarded as a full subcategory of the overcategory $\C_{/X}$, and thus it comes equipped with a forgetful functor $S \to \C$.
  The sieve $S$ induces a cocone over this functor, whose cocone point is $X$, and whose coprojections are the morphisms in $S$. This cocone will be used later.
\end{remark}

\begin{proposition}
  \href{https://github.com/leanprover-community/mathlib4/blob/743032e7ead097fb3e8ae5cd02d29cdd8899161c/Mathlib/CategoryTheory/Sites/IsSheafFor.lean#L619-L634}{\faExternalLink}
  Let $X$ be an object in $\C$ and let $S$ be a presieve on $X$. A presheaf $F$ is a sheaf for $S$ if and only if it is a sheaf for $\langle S \rangle$.
\end{proposition} 
\begin{proof}
  See~\cite[Lemma C.2.1.3]{elephant} or \mathlib.
\end{proof}

\begin{definition}\label{def:pullback_sieve}
  \href{https://github.com/leanprover-community/mathlib4/blob/743032e7ead097fb3e8ae5cd02d29cdd8899161c/Mathlib/CategoryTheory/Sites/Sieves.lean#L512-L518}{\faExternalLink}
  The \emph{pullback} of a sieve $S=(g_i \colon Y_i\to Y)_{i\in I}$ on $Y$ along a morphism $f \colon X \to Y$ is the sieve on $X$ consisting of all morphisms $g\colon Y_i\to X$ (for $i\in I$) such that $f \circ g \in S$. It is denoted $f^*S$.

\end{definition}

\begin{definition}\label{def:grothendieck-topology}
  \href{https://github.com/leanprover-community/mathlib4/blob/743032e7ead097fb3e8ae5cd02d29cdd8899161c/Mathlib/CategoryTheory/Sites/Grothendieck.lean#L63-L88}{\faExternalLink}
  A \emph{Grothendieck topology} on $\C$ is the datum of a set of sieves on each object $X$ of $\C$, called \emph{covering sieves} satisfying the following properties:
  \begin{list_named_conditions}{GT}
      \item The \emph{top sieve} --- consisting of all morphisms in $\C$ with target $X$ --- is a covering sieve on $X$. 
    \label{point:GroTop_top}
      \item For every covering sieve $S$ on $Y$ and every morphism $f \colon X \to Y$, the pullback $f^*S$ is a covering sieve on $X$. 
    \label{point:GroTop_pullback}
      \item Given a covering sieve $S$ on $Y$, suppose another sieve $R$ on $Y$ satisfies the property that for every $f : X \to Y \in S$, $f^*R$ is a covering sieve on $X$. Then $R$ is also a covering sieve on $Y$.
    \label{point:GroTop_closed}
  \end{list_named_conditions}

\end{definition}

\begin{lemma}\label{lem:superset_covering}
  \href{https://github.com/leanprover-community/mathlib4/blob/743032e7ead097fb3e8ae5cd02d29cdd8899161c/Mathlib/CategoryTheory/Sites/Grothendieck.lean#L140-L150}{\faExternalLink}
  Let $\T$ be a Grothendieck topology on $\C$, let $X$ be an object of $\C$, and $S$ and $R$ be two sieves on $X$ such that $S$ is contained in $R$ \textup{(}meaning that every morphism in $S$ is in $R$\textup{)}. If $S$ is a covering sieve for $\T$, then $R$ is a covering sieve as well. 
\end{lemma}
\begin{proof}
    By axiom~\ref{point:GroTop_closed}, it suffices to show that for every $f\colon Y \to X$ in $S$, $f^*R$ is a covering sieve of $Y$. By axiom~\ref{point:GroTop_top} it suffices to show that $f^*R$ contains every morphism to $Y$. So let $g \colon Z \to Y$ be a morphism. Since $f \circ g$ is in $S$, it is in $R$, meaning that $g$ is in $f^*R$, as desired.
%
\end{proof}

\begin{definition}\label{def:Groth_top_coverage}
  \href{https://github.com/leanprover-community/mathlib4/blob/743032e7ead097fb3e8ae5cd02d29cdd8899161c/Mathlib/CategoryTheory/Sites/Coverage.lean#L159-L172}{\faExternalLink}
  \href{https://github.com/leanprover-community/mathlib4/blob/743032e7ead097fb3e8ae5cd02d29cdd8899161c/Mathlib/CategoryTheory/Sites/Coverage.lean#L206-L239}{\faExternalLink}
  The coverage associated to a Grothendieck topology $\T$ is the coverage whose covering presieves are those whose associated sieve is a covering sieve in $\T$. 
  The Grothendieck topology generated by a coverage $\mathcal{S}$ is the intersection of all Grothendieck topologies whose associated coverage contains $\mathcal{S}$.
\end{definition}

Another definition of the Grothendieck topology~$\T$ generated by a coverage can be given in terms of a \emph{saturation} process. To define this, we start by ordering the collections of sieves on an object $X$ by objectwise inclusion; given a coverage $\mathcal{S}$, its \emph{saturation} is the smallest family $\bigr(C(X)\bigl)_{X\in\C}$ of collections of sieves satisfying:
\begin{list_named_conditions}{Sat}
    \item For every object $X$, the top sieve on $X$ is in $C(X)$.
		\label{point:Saturation_top}
    \item For every object $X$ and every covering presieve $S$ on $X$ in $\mathcal{S}$, we have $\langle S \rangle \in C(X)$.
		\label{point:Saturation_pullback}
    \item For every object $X$ and every pair $S, R$ of sieves on $X$ such that $S \in C(X)$ and such that for each $f \in S$ the pullback $f^*R$ belongs to $C(Y)$, we have that $R$ lies in $C(X)$.
		\label{point:Saturation_closed}
\end{list_named_conditions}
In terms of the dependent type theory underlying \Lean, requiring that this be ``the smallest family'' with a certain property is particularly handy, as it can be formalized in terms of \emph{inductive types}, a notion that lies at the very core of the foundational set-up of \Lean and therefore whose implementation and development is remarkably well integrated. 
This inductive construction is the one that is currently implemented in \mathlib as follows \href{https://github.com/leanprover-community/mathlib4/blob/743032e7ead097fb3e8ae5cd02d29cdd8899161c/Mathlib/CategoryTheory/Sites/Coverage.lean#L177-L187}{\faExternalLink}:

\begin{lstlisting}
inductive saturate (K : Coverage C) : (X : C) → Sieve X → Prop where
  | of (X : C) (S : Presieve X) (hS : S ∈ K X) : saturate K X (Sieve.generate S)
  | top (X : C) : saturate K X ⊤
  | transitive (X : C) (R S : Sieve X) :
    saturate K X R →
    (∀ ⦃Y : C⦄ ⦃f : Y ⟶ X⦄, R f → saturate K Y (S.pullback f)) →
    saturate K X S 
\end{lstlisting}

To prove that the saturation of $\mathcal{S}$ is in fact a Grothendieck topology, axioms~\ref{point:GroTop_top} and~\ref{point:GroTop_closed} follow at once from the defining properties~\ref{point:Saturation_top} and~\ref{point:Saturation_closed} of the saturation. Verifying property~\ref{point:GroTop_pullback} requires a bit more work and is achieved by applying the principle of induction on this inductive type.
The formalization of this property is \href{https://github.com/leanprover-community/mathlib4/blob/743032e7ead097fb3e8ae5cd02d29cdd8899161c/Mathlib/CategoryTheory/Sites/Coverage.lean#L206-L239}{\faExternalLink}:

\begin{lstlisting}
def toGrothendieck (K : Coverage C) : GrothendieckTopology C where
  sieves := saturate K
  top_mem' := .top
  pullback_stable' := by ... --the inductive proof mentioned above
  transitive' X S hS R hR := .transitive _ _ _ hS hR
\end{lstlisting}
 
It follows quite easily that the definition through saturations coincides with the one in Definition~\ref{def:Groth_top_coverage}, an equivalence whose proof is formalized in the theorem~\href{https://github.com/leanprover-community/mathlib4/blob/743032e7ead097fb3e8ae5cd02d29cdd8899161c/Mathlib/CategoryTheory/Sites/Coverage.lean#L271-L286}{\faExternalLink}:

\begin{lstlisting}
theorem toGrothendieck_eq_sInf (K : Coverage C) : toGrothendieck _ K =
    sInf {J | K ≤ ofGrothendieck _ J } := by ...
\end{lstlisting}

\begin{definition}
  \href{https://github.com/leanprover-community/mathlib4/blob/743032e7ead097fb3e8ae5cd02d29cdd8899161c/Mathlib/CategoryTheory/Sites/SheafOfTypes.lean#L70-L76}{\faExternalLink}
  Let $\T$ be a Grothendieck topology on $\C$. A presheaf $F \colon \C^{\op} \to \Set$ is a \emph{sheaf} for $\T$ if it is a sheaf for every covering sieve.
\end{definition}

\begin{proposition}\label{prop:Presieve.isSheaf_coverage}
  \href{https://github.com/leanprover-community/mathlib4/blob/743032e7ead097fb3e8ae5cd02d29cdd8899161c/Mathlib/CategoryTheory/Sites/Coverage.lean#L321-L394}{\faExternalLink}
  If a Grothendieck topology $\T$ is generated by a coverage, then a presheaf is a sheaf if and only if it is a sheaf for every covering presieve in the coverage. 
\end{proposition}
\begin{proof}
  A proof can be found in~\cite[Proposition C.2.1.9]{elephant}. 
  The proof that appears in \mathlib uses induction based on the inductive definition of the Grothendieck topology generated by a coverage discussed above.
  If one uses Definition~\ref{def:Groth_top_coverage} instead, a proof can be obtained by using the equivalence of this definition with the inductive construction.
\end{proof}

\section{Effective epimorphisms}\label{sec:effective}
In the literature, there are three related conditions on a morphism, designed to capture the property of surjectivity better than the standard notion of an epimorphism.
These are called \emph{strict}, \emph{regular} and \emph{effective} epimorphisms respectively; each property implies the previous one. However, each property requires more assumptions on the underlying category than the previous one, and when the assumptions to define \emph{effective epimorphism} hold, then strict implies effective. So, in a sense, these conditions are all equivalent. This is why it was decided to use the name \emph{effective} in \mathlib for the most generally applicable notion, usually called \emph{strict}. For a more precise explanation of this justification of terminology, see the text following Definition~\ref{EffectiveEpiFamily}.

In the category of topological spaces and the category of compact Hausdorff spaces, the effective epimorphisms are precisely the quotient maps. In the latter, the quotient maps are simply the continuous surjections, so the properties of being surjective, an epimorphism and an effective epimorphism all coincide (see Propositions~\ref{prop:TocCatEffectiveEpi} and~\ref{prop:CompHausEffectiveEpi}).

\begin{definition}\label{RegularEpi}
  \href{https://github.com/leanprover-community/mathlib4/blob/743032e7ead097fb3e8ae5cd02d29cdd8899161c/Mathlib/CategoryTheory/Limits/Shapes/RegularMono.lean#L183-L192}{\faExternalLink}
  A morphism $f \colon X \to B$ in a category $\C$ is a \emph{regular epimorphism} if it exhibits $B$ as a coequalizer of some pair of morphisms $g_1, g_2 \colon Z \to X$. 
\end{definition}

\begin{remark}
  \href{https://github.com/leanprover-community/mathlib4/blob/743032e7ead097fb3e8ae5cd02d29cdd8899161c/Mathlib/CategoryTheory/Limits/Shapes/RegularMono.lean#L214-L221}{\faExternalLink}
  If a regular epimorphism $f \colon X \to B$ has a kernel pair (meaning that the pullback $X\times_B X$ exists), then $B$ is the coequalizer of the two projections $X \times_B X \to X$.
\end{remark}

\begin{definition}\label{EffectiveEpi}
    \href{https://github.com/leanprover-community/mathlib4/blob/743032e7ead097fb3e8ae5cd02d29cdd8899161c/Mathlib/CategoryTheory/EffectiveEpi/Basic.lean#L42-L69}{\faExternalLink}
    A morphism $f \colon Y \to X$ in a category $\C$ is an \emph{effective epimorphism} if it satisfies the following condition: for every morphism $e$ that coequalizes every pair of parallel morphisms which $f$ coequalizes, there exists a unique morphism $d$ such that $d \circ f = e$:
    \[\begin{tikzcd}
        Z & Y & X \\
        && W.
        \arrow["{g_1}", shift left, from=1-1, to=1-2]
        \arrow["{g_2}"', shift right, from=1-1, to=1-2]
        \arrow["f", from=1-2, to=1-3]
        \arrow["e"', from=1-2, to=2-3]
        \arrow["{\exists!\,d}", dashed, from=1-3, to=2-3]
    \end{tikzcd}\]
\end{definition}

\begin{remark}\label{rmk:effetive_epi_epi}
  It is easy to check that if $f\colon Y\to X$ is an effective epimorphism, then it is an epimorphism. Indeed, given a diagram
  \[\begin{tikzcd}
    Y & X & W
    \arrow["{h_1}", shift left, from=1-2, to=1-3]
    \arrow["{h_2}"', shift right, from=1-2, to=1-3]
    \arrow["f", from=1-1, to=1-2]
\end{tikzcd}\]
such that $h_1\circ f=h_2\circ f$, observe that $h_1\circ f$ equalizes every pair of morphisms $g_1,g_2\colon Z\to Y$ equalized by~$f$. In particular, there is a unique map $d\colon X\to W$ such that $d\circ f = h_1\circ f$, and since $h_1$ and $h_2$ both satisfy this property, we deduce $h_1=h_2$.
\end{remark}

In \mathlib, the notion of effective epimorphism is implemented in two steps. First, we define a structure \lean{EffectiveEpiStruct} that contains the data required to be an effective epimorphism:

\begin{lstlisting}
structure EffectiveEpiStruct {X Y : C} (f : Y ⟶ X) where
  desc : ∀ {W : C} (e : Y ⟶ W),
    (∀ {Z : C} (g₁ g₂ : Z ⟶ Y), g₁ ≫ f = g₂ ≫ f → g₁ ≫ e = g₂ ≫ e) → (X ⟶ W)
  fac : ∀ {W : C} (e : Y ⟶ W)
    (h : ∀ {Z : C} (g₁ g₂ : Z ⟶ Y), g₁ ≫ f = g₂ ≫ f → g₁ ≫ e = g₂ ≫ e),
    f ≫ desc e h = e
  uniq : ∀ {W : C} (e : Y ⟶ W)
    (h : ∀ {Z : C} (g₁ g₂ : Z ⟶ Y), g₁ ≫ f = g₂ ≫ f → g₁ ≫ e = g₂ ≫ e)
    (m : X ⟶ W), f ≫ m = e → m = desc e h
\end{lstlisting}

The field \lean{desc} provides, given a morphism $e \colon Y \to W$ which coequalizes every morphism that $f$ coequalizes, the morphism $d \colon X \to W$; the field \lean{fac} is a proof that $d \circ f = e$; and the field \lean{uniq} is a proof that $d$ is unique. 

We then define a class \lean{EffectiveEpi}, which is a proposition saying that the type of \lean{EffectiveEpiStruct}'s associated to $f$ is nonempty\footnote{The fact that \lean{EffectiveEpi} is a class allows \Lean to use \emph{typeclass inference} to infer that a morphism is effective epimorphic in some cases: for example, in $\CompHaus$, given a morphism \lean{f} with an \lean{[Epi f]} instance, \Lean can automatically infer an instance \lean{EffectiveEpi f}. Moreover, the internal axiomatic of \Lean guarantees that two terms of a proposition are definitionally equal: in particular, two \emph{proofs} of non-emptiness of \lean{EffectiveEpiStruct f} automatically coincide, whereas producing explicit witnesses might lead to different outcomes, and that would often be troublesome.}:

\begin{lstlisting}
class EffectiveEpi {X Y : C} (f : Y ⟶ X) : Prop where
  effectiveEpi : Nonempty (EffectiveEpiStruct f)
\end{lstlisting}

\begin{definition}
    Given a family of morphisms $f=(f_i \colon X_i \to B)_{i \in I}$ and a pair of morphisms $g_{j_1} \colon Z \to X_{j_1}$ and $g_{j_2} \colon Z \to X_{j_2}$, we say that the family \emph{coequalizes $g_{j_1}$ and $g_{j_2}$} if $f_{j_1} \circ g_{j_1} = f_{j_2} \circ g_{j_2}$.
\end{definition}

\begin{definition}\label{EffectiveEpiFamily}
    \href{https://github.com/leanprover-community/mathlib4/blob/743032e7ead097fb3e8ae5cd02d29cdd8899161c/Mathlib/CategoryTheory/EffectiveEpi/Basic.lean#L102-L133}{\faExternalLink}
    A family of morphisms $(f_i \colon X_i \to B)_{i \in I}$ in a category $\C$ is \emph{effective epimorphic} if it satisfies the following universal property:
    
    Given any family $(e_i \colon X_i \to W)_{i \in I}$ coequalizing every pair of morphisms $g_i \colon Z \to X_i$, $g_j \colon Z \to X_j$ which $f$ coequalizes, there exists a unique morphism $d$ such that for all $i$, $d \circ f_i = e_i$:
    \[\begin{tikzcd}
        Z & {X_i} \\
        {X_j} & B \\
        && W
        \arrow["{g_i}", from=1-1, to=1-2]
        \arrow["{f_i}", from=1-2, to=2-2]
        \arrow["{g_j}"', from=1-1, to=2-1]
        \arrow["{f_j}"', from=2-1, to=2-2]
        \arrow["{e_i}", curve={height=-12pt}, from=1-2, to=3-3]
        \arrow["{e_j}"', curve={height=12pt}, from=2-1, to=3-3]
        \arrow["{\exists! d}"', dashed, from=2-2, to=3-3]
    \end{tikzcd}\] 
\end{definition}

The notion of effective epimorphic family is formalized in a similar two-step process where we first define 

\begin{lstlisting}
structure EffectiveEpiFamilyStruct {B : C} {α : Type*}
    (X : α → C) (π : (a : α) → (X a ⟶ B)) where
  desc : ∀ {W} (e : (a : α) → (X a ⟶ W)),
          (∀ {Z : C} (a₁ a₂ : α) (g₁ : Z ⟶ X a₁) (g₂ : Z ⟶ X a₂),
            g₁ ≫ π _ = g₂ ≫ π _ → g₁ ≫ e _ = g₂ ≫ e _) → (B ⟶ W)
  fac : ∀ {W} (e : (a : α) → (X a ⟶ W))
          (h : ∀ {Z : C} (a₁ a₂ : α) (g₁ : Z ⟶ X a₁) (g₂ : Z ⟶ X a₂),
            g₁ ≫ π _ = g₂ ≫ π _ → g₁ ≫ e _ = g₂ ≫ e _)
              (a : α), π a ≫ desc e h = e a
  uniq : ∀ {W} (e : (a : α) → (X a ⟶ W))
          (h : ∀ {Z : C} (a₁ a₂ : α) (g₁ : Z ⟶ X a₁) (g₂ : Z ⟶ X a₂),
            g₁ ≫ π _ = g₂ ≫ π _ → g₁ ≫ e _ = g₂ ≫ e _)
              (m : B ⟶ W), (∀ (a : α), π a ≫ m = e a) → m = desc e h
\end{lstlisting}

and then

\begin{lstlisting}
class EffectiveEpiFamily {B : C} {α : Type*} (X : α → C)
    (π : (a : α) → (X a ⟶ B)) : Prop where
  effectiveEpiFamily : Nonempty (EffectiveEpiFamilyStruct X π)
\end{lstlisting}

Definitions~\ref{EffectiveEpi} and~\ref{EffectiveEpiFamily} work in \emph{any} category; the morphism in question is not required to have a kernel pair. It is easy to see that if $f$ is a regular epimorphism, then it is an effective epimorphism. Conversely, if an effective epimorphism $f$ has a kernel pair, then it is a regular epimorphism (see \href{https://github.com/leanprover-community/mathlib4/blob/743032e7ead097fb3e8ae5cd02d29cdd8899161c/Mathlib/CategoryTheory/EffectiveEpi/RegularEpi.lean#L39-L61}{\faExternalLink}). This justifies the use of the terminology ``effective epimorphism''; 

We give some characterizations of effective epimorphic families.
For an object $W$ of $\C$, let $h_W$ denote the representable presheaf $h_W(X) = \Hom_{\C}(X,W)$.
\begin{lemma}\label{lem:effective-epi-family-characterizations}
  \href{https://github.com/leanprover-community/mathlib4/blob/743032e7ead097fb3e8ae5cd02d29cdd8899161c/Mathlib/CategoryTheory/Sites/EffectiveEpimorphic.lean#L244-L255}{\faExternalLink}
  Let $(f_i\colon X_i\to B)$ be a family of morphisms in $\C$. 
  Let $S$ be the sieve generated by the set $(f_i)_{i\in I}$, regarded as a presieve.
  Then the following are equivalent:
  \begin{list_tfae}
    \item The family $(f_i)_i$ is effective epimorphic.\label{lem:eff_epi-fam-char:point_effepi}
    \item For every object $W$ of $\C$, the presheaf $h_W$ is a sheaf for $S$.\label{lem:eff_epi-fam-char:sheaf}
    \item The cocone in $\C$ corresponding to the sieve $S$ \textup{(}described in Remark~\ref{rem:cocone}\textup{)} is colimiting.\label{lem:eff_epi-fam-char:cocone}
  \end{list_tfae}
\end{lemma}
\begin{proof}
  \ref{lem:eff_epi-fam-char:point_effepi}~$\iff$~\ref{lem:eff_epi-fam-char:sheaf}: First of all, observe that~\ref{lem:eff_epi-fam-char:sheaf} is equivalent to $h_W$ being a sheaf for $(f_i)_i$. Moreover, the data of a compatible family (in the sense of Definition~\ref{def:compatible_family}) for $(f_i)_i$ is a family $(x_i \colon X_i \to W)_i$ that coequalizes every pair of morphisms that $(f_i)_i$ coequalizes and an amalgamation for it is the morphism denoted $d$ in Definition~\ref{EffectiveEpiFamily}. The equivalence between~\ref{lem:eff_epi-fam-char:point_effepi} and~\ref{lem:eff_epi-fam-char:sheaf} follows.
  
  \ref{lem:eff_epi-fam-char:sheaf} $\implies$~\ref{lem:eff_epi-fam-char:cocone}: Suppose we have another cocone on the same functor, with cocone point $W$ and coprojections $x_f : X \to W$ for any $f : X \to B$ contained in $S$. 
We will now prove that this is precisely the data of a compatible family for $S$. Indeed, if $f \colon X \to B$ and $f' \colon X' \to B$ are in $S$, and the square 
  \[\begin{tikzcd}
    Y & {X'} \\
    X & B
    \arrow["{g'}", from=1-1, to=1-2]
    \arrow["g"', from=1-1, to=2-1]
    \arrow["f"', from=2-1, to=2-2]
    \arrow["{f'}", from=1-2, to=2-2]
  \end{tikzcd}\]
  commutes, then $f \circ g = f' \circ g' \in S$ because of the downwards closed property of sieves. We have coprojections $x_f \colon X \to W$, $x_{f'} \colon X' \to W$ and $x_{f \circ g} = x_{f' \circ g'} \colon Y \to W$ of the cocone with cocone point $W$, which satisfy 
  \[
    x_{f'} \circ g' = x_{f' \circ g'} = x_{f \circ g} = x_f \circ g
  \] 
  which is what we wanted. The unique amalgamation given by~\ref{lem:eff_epi-fam-char:sheaf} gives the unique cocone morphism required to satisfy the universal property of the colimit.
  
  \ref{lem:eff_epi-fam-char:cocone} $\implies$~\ref{lem:eff_epi-fam-char:point_effepi}: Given a family $(e_i\colon X_i\to W)$ that coequalizes any pair of morphisms $g_i\colon Z\to X_{i}$, $g_j\colon Z\to X_{j}$ that is coequalized by $f$, we obtain a cone over $S$ with cone point~$W$ as follows: recall that $S$ is generated by the~$(f_i)_i$, and thus the morphisms in~$S$ are precisely those which factor through~$f_i$ for some~$i$. 
  Thus, for each morphism $g\colon Y\to B$ in $S$, we may write $g = f_i\circ h$ for some $i$, and set $w_g:=e_i\circ h$ --- this is well-defined by the assumption on $(e_i)_i$
  We get the desired map $d\colon B\to W$ by the universal property of colimits.
\end{proof}

\begin{lemma}\label{lem:effectiveEpi_of_family}
  \href{https://github.com/leanprover-community/mathlib4/blob/743032e7ead097fb3e8ae5cd02d29cdd8899161c/Mathlib/CategoryTheory/EffectiveEpi/Coproduct.lean#L39-L52}{\faExternalLink}
  Let $(\pi_i \colon X_i \to B)_{i \in I}$ be an effective epimorphic family in $\C$, such that the coproduct of $(X_i)_i$ exists. The map 
  \[
    \pi \colon\coprod_i X_i \longrightarrow B  
  \]
  induced by $(\pi_i)_i$ is an effective epimorphism.
\end{lemma}
\begin{proof}
  Let $\iota_i \colon X_i \to \coprod_i X_i$ denote the coprojections of the coproduct.
  Let $e \colon \coprod_i X_i \to W$ be a morphism which coequalizes every pair of morphisms that $\pi$ coequalizes. It is clear that the family $(e \circ \iota_i)_{i \in I}$ coequalizes every pair $g_i \colon Z \to X_i$, $g_j \colon Z \to X_j$ that $(\pi_i)_{i \in I}$ coequalizes. It is easy to see that the morphism $d \colon B \to W$ obtained from the universal property of the effective epimorphic family gives the universal property of effective epimorphisms for $\pi$.
\end{proof}

\begin{lemma}\label{lem:effectiveEpiFamily_of_effectiveEpi}
  \href{https://github.com/leanprover-community/mathlib4/blob/743032e7ead097fb3e8ae5cd02d29cdd8899161c/Mathlib/CategoryTheory/EffectiveEpi/Coproduct.lean#L95-L117}{\faExternalLink}
  Let $(\pi_i \colon X_i \to B)_{i \in I}$ be a family of morphisms in $\C$. Suppose that
  \begin{list_conditions}
    \item All coproducts and pullbacks appearing in~\ref{pt:lemma_eff_of_eff:epi} exist.\label{pt:lemma_eff_of_eff:coproducts}
    \item For every object $Z$ and every morphism
  \[
    g \colon Z \longrightarrow \coprod_{i} X_i,  
  \]
  the induced map 
  \[
    i(g) := \coprod_i Z \times_{\coprod_i X_i} X_i \longrightarrow Z  
  \]
  is an epimorphism.\label{pt:lemma_eff_of_eff:epi}
\item  The map 
  \[
    \pi \colon \coprod_i X_i \longrightarrow B  
  \]
  induced by $(\pi_i)_i$ is an effective epimorphism.\label{pt:lemma_eff_of_eff:effective_epi}
  \end{list_conditions}
  Then $(\pi_i)_i$ is an effective epimorphic family.
\end{lemma}
\begin{proof}
  Let $(e_i \colon X_i \to Z)_{i \in I}$ be a family which coequalizes every pair of morphisms $g_i \colon Z \to X_i$, $g_j \colon Z \to X_j$ which $(\pi_i)_i$ coequalizes. We need to show that there exists a unique $d \colon B \to Z$ such that for all such $g_i, g_j$, we have $d \circ g_i = d \circ g_j$. To obtain this, we will apply the property that $\pi$ is an effective epimorphism to the induced morphism $e \colon \coprod_i X_i \to Z$. To be able to do this, we need to check that $e$ coequalizes every pair of morphisms which $\pi$ coequalizes. 

  Let $f_1, f_2 \colon Z \to \coprod_i X_i$ be given and suppose that $\pi \circ f_1 = \pi \circ f_2$. We want to show that $e \circ f_1 = e \circ f_2$. Applying the fact that $i(f_1)$ is an epimorphism, it suffices to prove that 
  \[
    e \circ f_1 \circ i(f_1) = e \circ f_2 \circ i(f_1).
  \] 
  This identity can be checked on each component of the coproduct $\coprod_i Z \times_{\coprod_i X_i} X_i$. In other words, we need to show that for every $a \in I$, 

  \[
    e \circ f_1 \circ i(f_1) \circ \iota_a = e \circ f_2 \circ i(f_1) \circ \iota_a,
  \]
  where 
  \[
    \iota_a \colon Z \times_{\coprod_i X_i} X_a \longrightarrow \coprod_i Z \times_{\coprod_i X_i} X_i
  \]
  denotes the coprojection. One easily checks that
  \[
    i(f_1) \circ \iota_a \colon Z \times_{\coprod_i X_i} X_a \longrightarrow Z
  \]
  is simply the first projection map in the pullback, which we denote by $p_1$. We thus need to show that 
  \[
    e \circ f_1 \circ p_1 = e \circ f_2 \circ p_1.
  \]
  The left-hand side simplifies to $e_a \circ p_2$, where
  \[
    p_2 \colon Z \times_{\coprod_i X_i} X_a \longrightarrow X_a
  \]
  denotes the second projection in the pullback. 

  Now it again suffices to prove the equality after precomposition with the epimorphism $i(f_2 \circ p_1)$, i.e. to show that 
  \[
    e_a \circ p_2 \circ i(f_2 \circ p_1) = e \circ f_2 \circ p_1 \circ i(f_2 \circ p_1).
  \]
  Again we can check this equality on the components of the coproduct $\coprod_{b}\left( Z \times_{\coprod_i X_i} X_a \right) \times_{\coprod_i X_i} X_b$, and similarly to above, this reduces to showing that for every $b \in I$, 
  \[
    e_a \circ g_a = e_b \circ g_b,
  \]
  where 
  \[
    g_a \colon \left( Z \times_{\coprod_i X_i} X_a \right) \times_{\coprod_i X_i} X_b \longrightarrow X_a
  \]
  is the first projection followed by the second projection, and
  \[
    g_b \colon  \left( Z \times_{\coprod_i X_i} X_a \right) \times_{\coprod_i X_i} X_b \longrightarrow X_b
  \]
  is the second projection. Doing the same manipulation on the equality $\pi \circ f_1 = \pi \circ f_2$, we see that $g_a, g_b$ is a pair of morphisms that the family $(\pi_i)_i$ coequalizes. By assumption, the family $(e_i)_i$ coequalizes it as well. This means that $e \circ f_1 = e \circ f_2$ and we obtain the unique $d \colon B \to Z$ we wanted.
\end{proof}

Propositions~\ref{prop:TocCatEffectiveEpi} and~\ref{prop:CompHausEffectiveEpi} provide an explicit description of effective epimorphisms in the categories of topological spaces, compact Hausdorff spaces, profinite spaces, and Stonean spaces. Both results ultimately rely on the observation that epimorphisms in these four categories are surjective, and we start with this result:
\begin{lemma}\label{lemma:epi_iff_surj}
  \href{https://github.com/leanprover-community/mathlib4/blob/743032e7ead097fb3e8ae5cd02d29cdd8899161c/Mathlib/Topology/Category/TopCat/EpiMono.lean#L28-L35}{\faExternalLink}
  \href{https://github.com/leanprover-community/mathlib4/blob/743032e7ead097fb3e8ae5cd02d29cdd8899161c/Mathlib/Topology/Category/CompHaus/Basic.lean#L335-L376}{\faExternalLink}
  \href{https://github.com/leanprover-community/mathlib4/blob/743032e7ead097fb3e8ae5cd02d29cdd8899161c/Mathlib/Topology/Category/Profinite/Basic.lean#L364-L399}{\faExternalLink}
  \href{https://github.com/leanprover-community/mathlib4/blob/743032e7ead097fb3e8ae5cd02d29cdd8899161c/Mathlib/Topology/Category/Stonean/Basic.lean#L183-L214}{\faExternalLink}
  Let $\C$ be any of the categories $\Top$, $\CompHaus$, $\Profinite$ or $\Stonean$. Then epimorphisms in $\C$ are surjective \textup{(}continuous\textup{)} maps.
\end{lemma}
\begin{proof}
  Note first that one direction is clear, because a surjective morphism in any concrete category is an epimorphism. Now let $f\colon Y \to X$ be a morphism in $\C$.

  When $\C=\Top$ the result is very well known: suppose $f$ is an epimorphism and consider the diagram
  \[
    \begin{tikzcd}
      Y & X & \{0,1\}^\flat
      \arrow["{f}", from=1-1, to=1-2]
      \arrow["\chi", shift left, from=1-2, to=1-3]
      \arrow["{e_1}" ', shift right, from=1-2, to=1-3]
  \end{tikzcd}
  \]
  where $\{0,1\}^\flat$ denotes the set $\{0,1\}$ endowed with the indiscrete topology, where $\chi$ is the characteristic function of $\image(f)$ and where $e_1$ is the constant map with image $1$. Clearly, $\chi \circ f=e_1\circ f$ and when $f$ is an epimorphism this implies that $\chi=e_1$, which is the statement $\image(f)=X$.

  When $\C=\CompHaus$, the above proof breaks down because $\{0,1\}^\flat$ is not in $\C$. But since spaces in $\C$ are normal, we can argue as follows: the subspace $\image(f)\subseteq X$ is compact, hence closed. Suppose that $f$ is not surjective, and let $x\notin\image(f)$: by Urysohn's lemma, there is a continuous function
  $\theta\colon X\to [0,1]$ such that $\theta(x)=0$ and $\theta(\image(f))=1$. Denote by $e_1\colon X\to [0,1]$ the constant function with image $1$: then $e_1\neq \theta$ and yet $f\circ \theta=f\circ e_1$ showing that $f$ is not an epimorphism.

  When $\C=\Profinite$ or $\C=\Stonean$ the above proof breaks down, because the unit interval is not in $\C$. But the argument for $\Top$ can be adapted by replacing the indiscrete space $\{0,1\}^\flat$ with the \emph{discrete} space $\{0,1\}^\delta$, which is in $\C$. First, observe that, given any topological space $Z$ and a clopen $U\subseteq Z$, the characteristic function $\chi_U$ is continuous.  Moreover, since every object in $\C$ is totally disconnected, its topology admits a basis of open neighbourhoods that are clopen sets \href{https://github.com/leanprover-community/mathlib4/blob/743032e7ead097fb3e8ae5cd02d29cdd8899161c/Mathlib/Topology/Separation.lean#L2570-L2577}{\faExternalLink}. Now suppose $f$ is not surjective, and let $x \notin \image(f)$. Since --- as before --- $\image(f)$ is closed, there exists an open neighbourhood~$V$ of~$x$ contained in the complement~$\image(f)^c$ and we can find a clopen neighbourhood~$U\subseteq V$ such that $x\in U$ and $U\cap \image(f)=\emptyset$. Consider the diagram in $\C$
\[
    \begin{tikzcd}
      Y & X & \{0,1\}^\delta
      \arrow["{f}", from=1-1, to=1-2]
      \arrow["\chi_U", shift left, from=1-2, to=1-3]
      \arrow["{e_0}" ', shift right, from=1-2, to=1-3]
  \end{tikzcd}
  \]
where $e_0$ is the constant function with value $0$. Now $\chi_U\neq e_0$, as can be seen by evaluating them on $x$, yet $\chi_U\circ f=e_0\circ f$ since $U\cap \image(f)=\emptyset$. This shows that $f$ is not an epimorphism.
\end{proof}
\begin{lemma}\label{quotient_map_is_effective_epi}
  Let $\C$ be a full subcategory of $\Top$ and let $f\colon Y\to X$ be a morphism in $\C$ which is a quotient map. Then $f$ is an effective epimorphism in $\C$. 
\end{lemma}
\begin{proof}
  Suppose that $e\colon Y\to Z$ equalizes every morphism that $f$ equalizes. This means that for every pair of points $y_1,y_2\in Y$ such that $f(y_1)=f(y_2)$, we have $e(y_1)=e(y_2)$, as can be seen by considering the parallel morphisms $e_{y_1},e_{y_2}\colon Y\to Y$ sending everything to $y_1$ and to $y_2$, respectively. The universal property of the quotient topology on $X$ provides the existence of a unique continuous $d\colon X\to Z$ such that $d\circ f = e$, showing that $f$ is an effective epimorphism.
\end{proof}
\begin{proposition}\label{prop:TocCatEffectiveEpi}
  \href{https://github.com/leanprover-community/mathlib4/blob/743032e7ead097fb3e8ae5cd02d29cdd8899161c/Mathlib/Topology/Category/TopCat/EffectiveEpi.lean#L52-L75}{\faExternalLink}
  The effective epimorphisms in $\Top$ are the quotient maps.
\end{proposition}
\begin{proof}
  A quotient map is an effective epimorphism in $\Top$ by Lemma~\ref{quotient_map_is_effective_epi}.

  In the other direction, let $f\colon Y\to X$ be an effective epimorphism in $\Top$. By Remark~\ref{rmk:effetive_epi_epi} and Lemma~\ref{lemma:epi_iff_surj}, $f$ is surjective and we are simply left to prove that in this situation $X$ is endowed with the quotient topology, namely the final topology induced by $f$. Denote by $\widehat{X}$ the space whose underlying set coincides with $X$, but endowed with the final topology induced by $f$, so that the identity map $i\colon \widehat{X}\to X$ is continuous. In the diagram
  \[\begin{tikzcd}
    Y&X\\
    &\widehat{X}
    \arrow["f", from=1-1, to=1-2]
    \arrow["i", from=2-2, to=1-2]
    \arrow["\widehat{f}=f" ', from=1-1, to=2-2]
    \arrow["d" yshift=-1ex, dashed, bend left = 35, from=1-2, to=2-2]
  \end{tikzcd} \]
  the morphism $\widehat{f}$ equalizes every pair of morphisms equalized by $f$, so there exists a unique \emph{continuous} map $d\colon X\to\widehat{X}$ making the diagram commute. It follows that $d$ is induced by the identity, showing that $X$ is homeomorphic to $\widehat{X}$, as required.
\end{proof}

\begin{proposition}\label{prop:CompHausEffectiveEpi}
  \href{https://github.com/leanprover-community/mathlib4/blob/743032e7ead097fb3e8ae5cd02d29cdd8899161c/Mathlib/Topology/Category/CompHaus/EffectiveEpi.lean#L72-L85}{\faExternalLink}
  \href{https://github.com/leanprover-community/mathlib4/blob/743032e7ead097fb3e8ae5cd02d29cdd8899161c/Mathlib/Topology/Category/Profinite/EffectiveEpi.lean#L69-L82}{\faExternalLink}
  \href{https://github.com/leanprover-community/mathlib4/blob/743032e7ead097fb3e8ae5cd02d29cdd8899161c/Mathlib/Topology/Category/Stonean/EffectiveEpi.lean#L62-L75}{\faExternalLink}
  The effective epimorphisms in $\CompHaus,\Profinite$ and in $\Stonean$ are the \textup{(}continuous\textup{)} surjections.
\end{proposition}
\begin{proof}
  Let $\C$ be any of the categories $\CompHaus,\Profinite$ or $\Stonean$ and let $f\colon Y\to X$ be an effective epimorphism in $\C$. Combining Remark~\ref{rmk:effetive_epi_epi} and Lemma~\ref{lemma:epi_iff_surj}, yields that $f$ is a continuous surjection.

  In the other direction, consider a continuous surjection $f\colon Y\to X$ in $\C$. Since the objects of $\C$ are compact Hausdorff spaces, $f$ is also a closed map and hence a quotient map, and thus an effective epimorphism by Lemma~\ref{quotient_map_is_effective_epi}.
\end{proof}

\section{Three Grothendieck topologies}\label{sec:topologies}

\subsection{The regular topology}
\begin{definition}\label{def:preregular_cat}
  \href{https://github.com/leanprover-community/mathlib4/blob/743032e7ead097fb3e8ae5cd02d29cdd8899161c/Mathlib/CategoryTheory/Sites/Coherent/Basic.lean#L84-L104}{\faExternalLink}
  A category $\C$ is \emph{preregular} if the collection of presieves consisting of single effective epimorphisms forms a coverage. 
  In other words, if for every effective epimorphism $g \colon Z \to Y$ and every morphism $f \colon X \to Y$, there exists an effective epimorphism $h \colon W \to X$ and a morphism $i \colon W \to Z$ such that the following diagram commutes:
  \[\begin{tikzcd}
    W & Z \\
    X & Y
    \arrow["i", from=1-1, to=1-2]
    \arrow["g", two heads, from=1-2, to=2-2]
    \arrow["h"', two heads, from=1-1, to=2-1]
    \arrow["f"', from=2-1, to=2-2]
  \end{tikzcd}\]
  In this case, we call this coverage the \emph{regular coverage} on $\C$, and the Grothendieck topology generated by this coverage is called the \emph{regular topology} on $\C$.

\end{definition}

In \mathlib, we define a predicate \lean{Preregular} \href{https://github.com/leanprover-community/mathlib4/blob/743032e7ead097fb3e8ae5cd02d29cdd8899161c/Mathlib/CategoryTheory/Sites/Coherent/Basic.lean#L84-L104}{\faExternalLink} on categories:

\begin{lstlisting}
class Preregular : Prop where
  exists_fac : ∀ {X Y Z : C} (f : X ⟶ Y) (g : Z ⟶ Y) [EffectiveEpi g],
    (∃ (W : C) (h : W ⟶ X) (_ : EffectiveEpi h) (i : W ⟶ Z), i ≫ g = h ≫ f)
\end{lstlisting}

Then the definition of the regular topology follows \href{https://github.com/leanprover-community/mathlib4/blob/743032e7ead097fb3e8ae5cd02d29cdd8899161c/Mathlib/CategoryTheory/Sites/Coherent/Basic.lean#L106-L129}{\faExternalLink}:

\begin{lstlisting}
def regularCoverage [Preregular C] : Coverage C where
  covering B := { S | ∃ (X : C) (f : X ⟶ B), S = Presieve.ofArrows (fun (_ : Unit) ↦ X)
    (fun (_ : Unit) ↦ f) ∧ EffectiveEpi f }
  pullback := by ...

def regularTopology [Preregular C] : GrothendieckTopology C :=
  Coverage.toGrothendieck _ <| regularCoverage C
\end{lstlisting}

\subsection{The extensive topology}
\begin{definition}\label{def:finitary extensive}
    \href{https://github.com/leanprover-community/mathlib4/blob/743032e7ead097fb3e8ae5cd02d29cdd8899161c/Mathlib/CategoryTheory/Extensive.lean#L90-L96}{\faExternalLink}
    A category $\C$ is \emph{finitary extensive} if it satisfies the following properties:
    \begin{list_conditions}
        \item $\C$ has finite coproducts.\label{pt:def_finitary_extensive:coproducts}
        \item $\C$ has pullbacks along coprojections of finite coproducts.\label{pt:def_finitary_extensive:pullbacks}
        \item Every commutative diagram 
        \[\begin{tikzcd}
            {Z_1} & Z & {Z_2} \\
            X & {X\coprod Y} & Y
            \arrow[from=1-1, to=2-1]
            \arrow[from=1-1, to=1-2]
            \arrow[from=2-1, to=2-2]
            \arrow[from=1-2, to=2-2]
            \arrow[from=2-3, to=2-2]
            \arrow[from=1-3, to=1-2]
            \arrow[from=1-3, to=2-3]
        \end{tikzcd}\]
        consists of two pullback squares if and only if the top row is a coproduct diagram.\label{pt:def_finitary_extensive:diagram}
    \end{list_conditions}    

\end{definition}

\begin{remark}
    Our definition of finitary extensive category is precisely~\cite[Definition~2.1 and Proposition~2.2]{carboni}.
\end{remark}

\mathlib already had the predicate \lean{FinitaryExtensive} on categories:

\begin{lstlisting}
class FinitaryExtensive (C : Type u) [Category.{v} C] : Prop where
  [hasFiniteCoproducts : HasFiniteCoproducts C]
  [hasPullbacksOfInclusions : HasPullbacksOfInclusions C]
  van_kampen' : ∀ {X Y : C} (c : BinaryCofan X Y), IsColimit c → IsVanKampenColimit c
\end{lstlisting}

The field \lean{van_kampen'} is condition~\ref{pt:def_finitary_extensive:diagram} in Definition~\ref{def:finitary extensive}. 

\begin{proposition}\label{prop:extensiveCoverage_is_coverage}
    Let $\C$ be a finitary extensive category. The collection of finite families $(X_i \to X)_{i \in I}$ exhibiting~$X$ as a coproduct of the family $(X_i)_{i\in I}$, forms a coverage.
\end{proposition}
\begin{proof}
    The axioms of a finitary extensive category ensure that the required property holds, namely that given a morphism $f \colon X \to Y$ and a finite family of morphisms $(g_i \colon Y_i \to Y)_{i \in I}$, the family $(X \times_Y Y_i \to X)_{i \in I}$ exhibits $X$ as a coproduct of the family $(X \times_Y Y_i)_{i \in I}$. This has been formalized in \mathlib~\href{https://github.com/leanprover-community/mathlib4/blob/743032e7ead097fb3e8ae5cd02d29cdd8899161c/Mathlib/CategoryTheory/Sites/Coherent/Basic.lean#L131-L150}{\faExternalLink}, but it appears \emph{ibid}. as a definition: this is because the proof that the collection is a coverage is part of the definition in question.
\end{proof}

\begin{definition}
  \href{https://github.com/leanprover-community/mathlib4/blob/743032e7ead097fb3e8ae5cd02d29cdd8899161c/Mathlib/CategoryTheory/Sites/Coherent/Basic.lean#L152-L156}{\faExternalLink}  
  Let $\C$ be a finitary extensive category. The coverage defined in Proposition~\ref{prop:extensiveCoverage_is_coverage} is called the \emph{extensive coverage} on $\C$, and the Grothendieck topology generated by this coverage is called the \emph{extensive topology} on $\C$. 

\end{definition}

In \mathlib, we define the extensive topology as follows \href{https://github.com/leanprover-community/mathlib4/blob/743032e7ead097fb3e8ae5cd02d29cdd8899161c/Mathlib/CategoryTheory/Sites/Coherent/Basic.lean#L131-L156}{\faExternalLink}:

\begin{lstlisting}
def extensiveCoverage [FinitaryPreExtensive C] : Coverage C where
  covering B := { S | ∃ (α : Type) (_ : Finite α) (X : α → C) (π : (a : α) → (X a ⟶ B)), S = Presieve.ofArrows X π ∧ IsIso (Sigma.desc π) }
  pullback := by ...

def extensiveTopology [FinitaryPreExtensive C] : GrothendieckTopology C :=
  Coverage.toGrothendieck _ <| extensiveCoverage C
\end{lstlisting}

Note that the definition of the extensive coverage and extensive topology only requires an assumption \lean{[FinitaryPreExtensive C]}. This condition is slightly weaker than \lean{FinitaryExtensive}, but the difference is unimportant. For the characterization of sheaves for the extensive topology, the stronger condition is indeed required.

\subsection{The coherent topology}
\begin{definition}\label{def:coherent coverage}
  \href{https://github.com/leanprover-community/mathlib4/blob/743032e7ead097fb3e8ae5cd02d29cdd8899161c/Mathlib/CategoryTheory/Sites/Coherent/Basic.lean#L48-L63}{\faExternalLink}
  A category $\C$ is \emph{precoherent} if the collection of finite effective epimorphic families forms a coverage. In other words, if for any finite effective epimorphic family $(\pi_i \colon X_i\to B)_{i\in I}$ and any morphism $f \colon B'\to B$, there exists a finite effective epimorphic family $(\pi'_j \colon X'_j \to B')_{j\in I'}$, such that for each $j\in I'$, the composition $f\circ \pi_j'$ factors through $\pi_i$ for some $i\in I$.
  In this case, we call this coverage the \emph{coherent coverage} on $\C$, and the Grothendieck topology generated by this coverage is called the \emph{coherent topology} on $\C$.

\end{definition} 

In \mathlib, we define a predicate \lean{Precoherent} \href{https://github.com/leanprover-community/mathlib4/blob/743032e7ead097fb3e8ae5cd02d29cdd8899161c/Mathlib/CategoryTheory/Sites/Coherent/Basic.lean#L48-L63}{\faExternalLink} on categories:

\begin{lstlisting}
class Precoherent : Prop where
  pullback {B₁ B₂ : C} (f : B₂ ⟶ B₁) :
    ∀ (α : Type) [Finite α] (X₁ : α → C) (π₁ : (a : α) → (X₁ a ⟶ B₁)),
      EffectiveEpiFamily X₁ π₁ →
    ∃ (β : Type) (_ : Finite β) (X₂ : β → C) (π₂ : (b : β) → (X₂ b ⟶ B₂)),
      EffectiveEpiFamily X₂ π₂ ∧
      ∃ (i : β → α) (ι : (b :  β) → (X₂ b ⟶ X₁ (i b))),
      ∀ (b : β), ι b ≫ π₁ _ = π₂ _ ≫ f
\end{lstlisting}

Then the definition of the coherent topology follows \href{https://github.com/leanprover-community/mathlib4/blob/743032e7ead097fb3e8ae5cd02d29cdd8899161c/Mathlib/CategoryTheory/Sites/Coherent/Basic.lean#L65-L82}{\faExternalLink}:

\begin{lstlisting}
def coherentCoverage [Precoherent C] : Coverage C where
  covering B := { S | ∃ (α : Type) (_ : Finite α) (X : α → C) (π : (a : α) → (X a ⟶ B)),
    S = Presieve.ofArrows X π ∧ EffectiveEpiFamily X π }
  pullback := by ...

def coherentTopology [Precoherent C] : GrothendieckTopology C :=
  Coverage.toGrothendieck _ <| coherentCoverage C
\end{lstlisting}

\begin{remark}
  The notion of a precoherent category naturally arose through the formalization process, and was forced upon us by the ``\mathlib philosophy'' where definitions are often phrased in the most general way (see \S\ref{subsec:mathlib_generality}).
  Indeed, the condition that $\C$ is a precoherent category is precisely the minimal axiom needed to ensure that what we call the \emph{coherent coverage} above is indeed a coverage.
  A similar approach was taken to define the notion of a \emph{preregular} category.
  For example, we do not require the existence of pullbacks required in the definition of \emph{regular} and \emph{coherent} categories as in~\cite[A1.3]{elephant} and~\cite[A1.4]{elephant} respectively. 
  
  Due to our weaker assumptions, several of our results about the regular and coherent topology strengthen existing standard results. For example, \cite[Example C.2.1.12 (d)]{elephant} states that the coherent topology on a \emph{coherent} category is subcanonical, which we extend in Proposition~\ref{prop:coherent-subcanonical} below to \emph{precoherent} categories. The analogous statement for the regular topology on a regular category can be found in~\cite[Corollary B.3.6]{ultracategories}, and is extended to preregular categories in Proposition~\ref{prop:regular-subcanonical} below. In  Proposition~\ref{prop:mem-coherent-topology-iff} (respectively Lemma~\ref{lem:mem-regular-topology-iff}), we explicitly characterize the covering sieves in the coherent (respectively regular) topology on a precoherent (respectively preregular) category. Under stronger assumptions on the category, this result can be found in~\cite[Definition B.5.1 and Proposition B.5.2]{ultracategories}  (respectively in~\cite[C.2.1.12 (c)]{elephant}).
\end{remark}

\subsection{The coherent topology on a regular extensive category}
\begin{proposition}\label{prop:precoherent_of_preregular_extensive}
    \href{https://github.com/leanprover-community/mathlib4/blob/743032e7ead097fb3e8ae5cd02d29cdd8899161c/Mathlib/CategoryTheory/Sites/Coherent/Comparison.lean#L39-L54}{\faExternalLink}
    Let $\C$ be a category which is preregular and finitary extensive. Then $\C$ is precoherent.
\end{proposition}
\begin{proof}
    Since $\C$ is finitary extensive, Lemmas~\ref{lem:effectiveEpi_of_family} and~\ref{lem:effectiveEpiFamily_of_effectiveEpi} imply that finite effective epimorphic families in $\C$ are precisely those which induce an effective epimorphism on the coproduct. 

    Let $(f_i \colon X_i \to X)_{i \in I}$ be a finite effective epimorphic family and let $g \colon Y \to X$ be a morphism. 
    Since the morphism $\coprod_i X_i \to X$ is an effective epimorphism, the fact that $\C$ is preregular ensures the existence of a diagram 
    \[\begin{tikzcd}
        Z & {\displaystyle{\coprod_i X_i}} \\
        Y & X
        \arrow["e", from=1-1, to=1-2]
        \arrow["f", from=1-2, to=2-2]
        \arrow["h"', from=1-1, to=2-1]
        \arrow["g"', from=2-1, to=2-2]
    \end{tikzcd}\]
    in which $h \colon Z \to Y$ is an effective epimorphism. 

    Now, the fact that $\C$ is extensive ensures that the family $(Z \times_{\coprod_i X_i} X_i \to Z)_{i \in I}$ exhibits $Z$ as a coproduct in the sense that the canonical map 
    \[ \coprod_i Z \times_{\coprod_i X_i} X_i \longrightarrow Z \]
    is an isomorphism.
    Therefore, the composition
    \[ \coprod_i Z \times_{\coprod_i X_i} X_i \longrightarrow Y \]
    is an effective epimorphism, and therefore the family $(Z \times_{\coprod_i X_i} X_i \to Y)_{i \in I}$ works as the desired effective epimorphic family.
\end{proof}

It is obvious that the union of two coverages is a coverage. This allows us to state:

\begin{proposition}\label{prop:extensive_regular_generate_coherent}
    \href{https://github.com/leanprover-community/mathlib4/blob/743032e7ead097fb3e8ae5cd02d29cdd8899161c/Mathlib/CategoryTheory/Sites/Coherent/Comparison.lean#L56-L94}{\faExternalLink}
    Let $\C$ be a category which is preregular and finitary extensive. The union of the regular and extensive coverages generates the coherent topology. 
\end{proposition}
\begin{proof}
Denote by $\T$ the topology generated by the union of the regular and extensive coverages. Note that the regular and extensive coverages are both contained in the coherent coverage, hence $\T$ is contained in the coherent topology, so it suffices to show that the coherent topology is contained in $\T$.

Let $X$ be an object of $\C$ and let $S$ be a covering sieve on $X$ for the coherent topology: in other words, $S$ is generated by a finite effective epimorphic family $(f_i\colon X_i \to X)_{i \in I}$. We want to show that $S$ is a $\T$-covering sieve. Denote by 
\[
    f\colon \coprod_{i\in I} X_i \longrightarrow X
\]
the map induced by the $f_i$ and for each $j \in I$, let
\[
    \iota_j\colon X_j \longrightarrow \coprod_{i\in I} X_i
\]
be the coprojection. For each $j$, 
\[
    f \circ \iota_j = f_j \in S, \text{ so }\iota_j \in f^*S.
\] 
Therefore, the sieve $T$ generated by the family $(\iota_i)_i$ is contained in $f^*S$. Since the presieve generated by the family $(\iota_i)_i$ is a covering presieve of the coproduct in the extensive coverage, $T$ is a $\T$-covering sieve and hence by Lemma~\ref{lem:superset_covering}, $f^*S$ is a $\T$-covering sieve of $\coprod_i X_i$. By Lemma~\ref{lem:effectiveEpi_of_family}, $f$ is an effective epimorphism, and hence the sieve $S_f$ generated by the singleton presieve $\{f\}$ is a $\T$-covering sieve. Now by axiom~\ref{point:GroTop_closed} for Grothendieck topologies, it suffices to show that $g^*S$ is a $\T$-covering sieve for every $g$ in $S_f$. Given such a $g = f \circ h$, we have $g^*S = h^*(f^*S)$ which is a $\T$-covering sieve because $f^*S$ is.
\end{proof}


\section{Sheaves} \label{sec:sheaves}

\subsection{Regular sheaves}
Let $\C$ be a preregular category (see Definition~\ref{def:preregular_cat}).

\begin{proposition}\label{prop:regular-subcanonical}
    \href{https://github.com/leanprover-community/mathlib4/blob/743032e7ead097fb3e8ae5cd02d29cdd8899161c/Mathlib/CategoryTheory/Sites/Coherent/RegularSheaves.lean#L260-L262}{\faExternalLink}
    The regular topology on $\C$ is subcanonical\footnote{A Grothendieck topology is called \emph{subcanonical} if every representable presheaf is a sheaf. By \emph{representable}, we mean a presheaf of the form $\Hom(-,W)$ for some object $W$ of $\C$.}. 
\end{proposition}
\begin{proof}
    We need to show that each presheaf of the form $h_W = \Hom(-, W)$ with $W$ an object of $\C$ is a sheaf.
    By Proposition~\ref{prop:Presieve.isSheaf_coverage}, it is enough to check that $h_W$ is a sheaf for each family consisting of a single effective epimorphism. Noting that a singleton family is effective epimorphic if and only if it consists of an effective epimorphism, this is now clear from Lemma~\ref{lem:effective-epi-family-characterizations}.
\end{proof}

\begin{lemma}\label{lem:mem-regular-topology-iff}
    A sieve in $\C$ is a covering sieve for the regular topology if and only if it contains an effective epimorphism.
\end{lemma}
\begin{proof}
    The proof is a simpler version of the proof of Proposition~\ref{prop:mem-coherent-topology-iff} below. The reader can easily take that proof and replace effective epimorphic families by effective epimorphisms, thereby filling in this proof (the key is to prove that effective epimorphisms in preregular categories are stable under composition). 
\end{proof}

\begin{lemma}\label{lemma:sheaf_iff_equalizerCondition}
    \href{https://github.com/leanprover-community/mathlib4/blob/743032e7ead097fb3e8ae5cd02d29cdd8899161c/Mathlib/CategoryTheory/Sites/Coherent/RegularSheaves.lean#L213-L222}{\faExternalLink}
    Suppose $\C$ has kernel pairs of effective epimorphisms. Then a presheaf $F$ on $\C$ is a sheaf for the regular topology if and only if for every effective epimorphism $\pi \colon X \to B$, the diagram 
    \begin{equation}\label{eq:kernel_pair_sheaf}\tag{EqCond}
        \begin{tikzcd}
        {F(B)} & {F(X)} & {F(X \times_B X)}
        \arrow["{F(\pi)}", from=1-1, to=1-2]
        \arrow[shift left, from=1-2, to=1-3]
        \arrow[shift right, from=1-2, to=1-3]
    \end{tikzcd}\end{equation}
    is an equalizer (the two parallel morphisms being given by the projections in the pullback). 
\end{lemma}
\begin{proof}
    This follows from the fact that a presheaf is a sheaf for the regular topology if and only if it is a sheaf for every family consisting of a single effective epimorphism, and the characterization (discussed in Remark~\ref{rmk:sheaf_with_pb}) of the sheaf condition in the case where the relevant pullbacks exist. 
\end{proof}

\begin{proposition}\label{prop:regularTopology.isSheaf_of_projective}
    \href{https://github.com/leanprover-community/mathlib4/blob/743032e7ead097fb3e8ae5cd02d29cdd8899161c/Mathlib/CategoryTheory/Sites/Coherent/RegularSheaves.lean#L233-L237}{\faExternalLink}
    Suppose every object in $\C$ is projective\footnote{An object $P$ is \emph{projective} if every morphism out of $P$ lifts along every epimorphism with the same target.}. Then every presheaf on $\C$ is a sheaf for the regular topology. 
\end{proposition}
\begin{proof}
    Since every object is projective, every sieve generated by an epimorphism is the top sieve, for which every presheaf is a sheaf. 
\end{proof}

\subsection{Extensive sheaves}
Let $\C$ be a finitary extensive category (see Definition~\ref{def:finitary extensive}).

\begin{proposition}\label{prop:isSheaf_iff_preservesFiniteProducts}
    \href{https://github.com/leanprover-community/mathlib4/blob/743032e7ead097fb3e8ae5cd02d29cdd8899161c/Mathlib/CategoryTheory/Sites/Coherent/ExtensiveSheaves.lean#L113-L132}{\faExternalLink}
    A presheaf on $\C$ is a sheaf with respect to the extensive topology if and only if it preserves finite products.
\end{proposition}
\begin{proof}
    This is proved in~\cite[Proposition B.4.5]{ultracategories} (there, the extensive topology is defined only for categories with pullbacks, but the proof remains valid in our setting since only pullbacks along coprojections of finite coproducts are used). Our formalization follows the same ideas used \emph{ibid}.
\end{proof}

\begin{proposition}\label{prop:extensive-subcanonical}
    \href{https://github.com/leanprover-community/mathlib4/blob/743032e7ead097fb3e8ae5cd02d29cdd8899161c/Mathlib/CategoryTheory/Sites/Coherent/ExtensiveSheaves.lean#L71-L73}{\faExternalLink}
      The extensive topology on $\C$ is subcanonical. 
  \end{proposition}
  \begin{proof}
    Since $\Hom(-, W)$ preserves limits, this follows from Proposition~\ref{prop:isSheaf_iff_preservesFiniteProducts}
  \end{proof}

  \begin{proposition}\label{prop:extensive_covering_sieves} 
    \href{https://github.com/leanprover-community/mathlib4/blob/743032e7ead097fb3e8ae5cd02d29cdd8899161c/Mathlib/CategoryTheory/Sites/Coherent/ExtensiveTopology.lean#L26-L57}{\faExternalLink}
    Let $X$ be an object of $\C$ and $S$ a sieve on $X$. Then $S$ is a covering sieve for the extensive topology on $\C$ if and only if it contains a family of morphisms $(f_i : X_i \to X)_{i \in I}$ which exhibit $X$ as a coproduct of the $X_i$.  
\end{proposition}
\begin{proof}
  The proof is a simpler version of the proof of Proposition~\ref{prop:mem-coherent-topology-iff} below. The reader can easily take that proof and replace effective epimorphic families by families of morphisms exhibiting the target as a coproduct, thereby filling in this proof.
\end{proof}

\subsection{Coherent sheaves}
\begin{proposition}\label{prop:coherent-subcanonical}
  \href{https://github.com/leanprover-community/mathlib4/blob/743032e7ead097fb3e8ae5cd02d29cdd8899161c/Mathlib/CategoryTheory/Sites/Coherent/CoherentSheaves.lean#L61-L63}{\faExternalLink}
	Let $\C$ be a precoherent category \textup{(}see Definition~\ref{def:coherent coverage}\textup{)}. The coherent topology on $\C$ is subcanonical. 
\end{proposition}
\begin{proof}
  We need to show that each presheaf of the form $h_W = \Hom(-, W)$ with $W$ an object of $\C$ is a sheaf.
  By Proposition~\ref{prop:Presieve.isSheaf_coverage}, it is enough to check that $h_W$ is a sheaf for each finite effective epimorphic family, and this follows from Lemma~\ref{lem:effective-epi-family-characterizations}.
\end{proof}

\begin{remark}
  If $\C$ is finitary extensive and preregular (and hence precoherent), then Proposition~\ref{prop:coherent-subcanonical} implies Proposition~\ref{prop:extensive-subcanonical} and Proposition~\ref{prop:regular-subcanonical}, because the coherent topology is finer than the extensive and regular one. 
  On the other hand, being precoherent might not in general imply being finitary extensive or preregular (for example, when $\C$ does not have finite coproducts) and this is why we proved Proposition~\ref{prop:extensive-subcanonical} and Proposition~\ref{prop:regular-subcanonical} separately.
\end{remark}

\begin{proposition}\label{prop:mem-coherent-topology-iff}
  \href{https://github.com/leanprover-community/mathlib4/blob/743032e7ead097fb3e8ae5cd02d29cdd8899161c/Mathlib/CategoryTheory/Sites/Coherent/CoherentTopology.lean#L78-L99}{\faExternalLink}
  Let $\C$ be a precoherent category. A sieve in $\C$ is a covering sieve for the coherent topology if and only if it contains a finite effective epimorphic family.
\end{proposition}

Before proving Proposition~\ref{prop:mem-coherent-topology-iff} we provide some preliminary results.
\begin{lemma}\label{lem:effective-epi-implies-coherent}
  \href{https://github.com/leanprover-community/mathlib4/blob/743032e7ead097fb3e8ae5cd02d29cdd8899161c/Mathlib/CategoryTheory/Sites/Coherent/CoherentTopology.lean#L24-L36}{\faExternalLink}
  If a sieve $S$ contains a finite effective epimorphic family, then $S$ is a covering sieve for the coherent topology.
\end{lemma}
\begin{proof}
  Let $(\pi_i \colon X_i \to X)_{i\in I}$ be a finite effective epimorphic family contained in $S$.
  By definition, the sieve $S_0$ generated by the family $(\pi_i)_{i\in I}$ is a covering sieve for the coherent topology, and since $S$ contains the family $(\pi_i)_{i\in I}$, it contains $S_0$. Lemma~\ref{lem:superset_covering} yields the conclusion.
\end{proof}

\begin{lemma}\label{lem:effective-epi-transitive}
  \href{https://github.com/leanprover-community/mathlib4/blob/743032e7ead097fb3e8ae5cd02d29cdd8899161c/Mathlib/CategoryTheory/Sites/Coherent/CoherentTopology.lean#L38-L68}{\faExternalLink}
	Assume that $\C$ is precoherent and that $(\pi_i\colon X_i\to B)_{i\in I}$ is a finite effective epimorphic family, and suppose that for each $i\in I$, we are given a finite effective epimorphic family $(\pi_{i,j}\colon Y_{i,j}\to X_i)_{j\in J_i}$. Then the induced collection $(\varpi_{i,j}=\pi_i\circ\pi_{i,j}\colon Y_{i,j}\to B)_{i\in I, j\in J_i}$ is an effective epimorphic family.  
\end{lemma}
\begin{proof}
  By Lemma~\ref{lem:effective-epi-family-characterizations}, a family is effective epimorphic if and only if for each object $W$ the presheaf~$h_W$ is a sheaf for the sieve generated by this family.
  Thus, since the coherent topology is subcanonical by Proposition~\ref{prop:coherent-subcanonical}, it is enough to show that the sieve $S$ generated by the family $(\varpi_{i,j})_{i\in I, j\in J_i}$ is a covering sieve for the coherent topology.
	 
  By and~\ref{point:GroTop_closed} of Definition~\ref{def:grothendieck-topology}, it is enough to check that $f^*S$ is a covering sieve for every map $f$ in the sieve generated by $(\pi_i)_{i\in I}$ (which is a covering sieve by Lemma~\ref{lem:effective-epi-implies-coherent}). 
  In fact, by~\ref{point:GroTop_pullback}, it is enough to check that each $\pi_i^*S$ is a covering sieve.
  Since $\pi_i^*S$ contains the finite effective epimorphic family $(\pi_{i,j})_{j\in I_j}$, it is a covering sieve for the coherent topology by Lemma~\ref{lem:effective-epi-implies-coherent}.
\end{proof}

\begin{proof}[Proof of Proposition~\ref{prop:mem-coherent-topology-iff}]
Let $\T$ denote the collection of sieves in $\C$ that contain a finite effective epimorphic family. By Lemma~\ref{lem:effective-epi-implies-coherent}, we know that $\T$ is contained in the coherent topology. Our goal is to show that they are equal, so it remains to show that $\T$ contains the coherent topology. By definition, the coherent topology is the smallest Grothendieck topology whose associated coverage contains the coherent coverage. Therefore, it suffices to show that
\begin{list_cases}
	\item the collection $\T$ forms a Grothendieck topology and \label{mem_coh_prop:is_grothendieck}
	\item the coverage associated to $\T$ contains the coherent coverage. \label{mem_coh_prop:contains_coherent}
\end{list_cases}
Once~\ref{mem_coh_prop:is_grothendieck} is established, point~\ref{mem_coh_prop:contains_coherent} is immediate from the definitions of $\T$ and of the associated coverage (Definition~\ref{def:Groth_top_coverage}). It remains to show~\ref{mem_coh_prop:is_grothendieck} by checking the conditions of Definition~\ref{def:grothendieck-topology}. Condition~\ref{point:GroTop_top} is immediate, since for every object $X$ of $\C$, the identity on $X$ forms a finite effective epimorphic family. Condition~\ref{point:GroTop_pullback} is a consequence of the precoherence assumption: Let $f\colon X\to Y$ be a morphism and let $S$ be a sieve on $Y$ that is contained in $\T$, i.e. that contains a finite effective epimorphic family $(\pi_i \colon Y_i\to Y)_{i\in I}$. Then the condition of being precoherent (see Definition~\ref{def:coherent coverage}) provides an effective epimorphic family $(\pi_j'\colon:X_j\to X)_{j\in I'}$ that is contained in the pullback sieve $f^*S$. Finally, we address~\ref{point:GroTop_closed}. Let $S,R$ be sieves on $Y$ with $S\in \T$ such that for every $f\colon X\to Y \in S$, the pullback sieve $f^*R$ is in $\T$. Then we have a finite effective epimorphic family $(f_i\colon X_i \to Y)_{i\in I}$ contained in $S$, and for each $i\in I$, a finite effective epimorphic family $(g_{i,j}\colon X_{i,j} \to X_i)_{j\in J_i}$ contained in $f_i^*R$. By Lemma~\ref{lem:effective-epi-transitive}, the finite family $(f_i\circ g_{i,j}\colon X_{i,j}\to Y)_{i\in I, j\in J_i}$ is effective epimorphic. By Definition~\ref{def:pullback_sieve} of the pullback sieve, the composition $f_i\circ g_{i,j} $ factors through some morphism in $R$, hence lies in~$R$ for each pair $(i,j)$. Thus the whole family $(f_i\circ g_{i,j})_{i\in I, j\in J_i}$ is contained in $R$, showing that $R\in \T$. This finishes the proof of Condition~\ref{point:GroTop_closed}.
\end{proof}

\begin{proposition}\label{prop:regular_extensive_sheaf}
    \href{https://github.com/leanprover-community/mathlib4/blob/743032e7ead097fb3e8ae5cd02d29cdd8899161c/Mathlib/CategoryTheory/Sites/Coherent/SheafComparison.lean#L235-L242}{\faExternalLink}
    Let $\C$ be a preregular, finitary extensive category with pullbacks of kernel pairs. A presheaf on $\C$ is a sheaf for the coherent topology if and only if it satisfies the equalizer condition~\eqref{eq:kernel_pair_sheaf} of Lemma~\ref{lemma:sheaf_iff_equalizerCondition}, and preserves finite products. 
\end{proposition}
\begin{proof}
    It is easy to see that a presheaf is a sheaf for the topology generated by a union of coverages if and only if it is a sheaf for every covering presieve of both coverages~\href{https://github.com/leanprover-community/mathlib4/blob/743032e7ead097fb3e8ae5cd02d29cdd8899161c/Mathlib/CategoryTheory/Sites/Coverage.lean#L424-L429}{\faExternalLink}. The result now follows by combining Proposition~\ref{prop:extensive_regular_generate_coherent} with Lemma~\ref{lemma:sheaf_iff_equalizerCondition} and Proposition~\ref{prop:isSheaf_iff_preservesFiniteProducts}. 
\end{proof}

\begin{proposition}\label{prop:regular_extensive_sheaf_projective}
    \href{https://github.com/leanprover-community/mathlib4/blob/743032e7ead097fb3e8ae5cd02d29cdd8899161c/Mathlib/CategoryTheory/Sites/Coherent/SheafComparison.lean#L244-L247}{\faExternalLink}
    Let $\C$ be a preregular, finitary extensive category in which every object is projective. A presheaf on $\C$ is a sheaf for the coherent topology if and only if it preserves finite products. 
\end{proposition}
\begin{proof}
    As in the proof of Proposition~\ref{prop:regular_extensive_sheaf}, the result follows by combining Proposition~\ref{prop:regularTopology.isSheaf_of_projective} with Proposition~\ref{prop:isSheaf_iff_preservesFiniteProducts}.
\end{proof}

\begin{proposition}\label{prop:sheafEquiv}
  \href{https://github.com/leanprover-community/mathlib4/blob/743032e7ead097fb3e8ae5cd02d29cdd8899161c/Mathlib/CategoryTheory/Sites/Coherent/SheafComparison.lean#L108-L117}{\faExternalLink}
  Let $\C$ be a category and let $F \colon \C \to \D$ be a fully faithful functor into a precoherent category $\D$ such that
\begin{list_conditions}
  \item $F$ preserves and reflects finite effective epimorphic families.\label{pt:propSheafEquiv:condReflects}
  \item  For every object $Y$ of $\D$, there exists an object $X$ of $\C$ and an effective epimorphism $F(X) \to Y$.\label{pt:propSheafEquiv:condExist}
\end{list_conditions}
Then the following holds:
\begin{list_cases}
\item $\C$ is precoherent.\label{pt:propSheafEquiv:precoherent}
\item Let $G$ be a sheaf for the coherent topology on $\D$. The presheaf $G \circ F^{\operatorname{op}}$ is a sheaf for the coherent topology on $\C$.\label{pt:propSheafEquiv:comp}
\item Precomposition with $F$ induces an equivalence between the categories of sheaves for the coherent topology on $\C$ and on $\D$.\label{pt:propSheafEquiv:equiv}
\end{list_cases}
\end{proposition}

Before proving Proposition~\ref{prop:sheafEquiv}, we need to fix some terminology and state some preliminary results. These preliminaries were already in \mathlib, and we simply state them here without proof. The results can be extracted from~\cite[Exposé III]{SGA4}, but the approach~\emph{ibid}.~differs slightly from the one in \mathlib.

\begin{definition}\label{def:functorPushforwardPullback} Let $\C$ and $\D$ be two categories, both endowed with a Grothendieck topology, and let $F\colon \C \to \D$ be a functor. Fix an object $X$ in $\C$ and an object $Y$ in $\D$.
  \begin{list_cases}
      \item Given a sieve $S$ on $X$, the \emph{functor-pushforward of $S$ along $F$} is the sieve $F_*S$ on $F(X)$ consisting of those morphisms $f \colon Y \to F(X)$ that factor through $F(g)$ for some morphism $g \colon Z \to X$ in $S$. \href{https://github.com/leanprover-community/mathlib4/blob/743032e7ead097fb3e8ae5cd02d29cdd8899161c/Mathlib/CategoryTheory/Sites/Sieves.lean#L689-L696}{\faExternalLink}
      \item Given a sieve $S$ on $F(X)$, the \emph{functor-pullback of $S$ along $F$} is the sieve $F^*S$ on $X$ consisting of those arrows $f$ such that $F(f)$ belongs to $S$. \href{https://github.com/leanprover-community/mathlib4/blob/743032e7ead097fb3e8ae5cd02d29cdd8899161c/Mathlib/CategoryTheory/Sites/Sieves.lean#L647-L657}{\faExternalLink}
      \item The \emph{$F$-image sieve} is the sieve $S^F_Y$ on $Y$ consisting of those morphisms to $Y$ that factor through an object of the form $F(X)$ for some $X$ in $\C$. \href{https://github.com/leanprover-community/mathlib4/blob/743032e7ead097fb3e8ae5cd02d29cdd8899161c/Mathlib/CategoryTheory/Sites/DenseSubsite.lean#L75-L80}{\faExternalLink}
  \end{list_cases}
\end{definition}
We omit the verification that Definition~\ref{def:functorPushforwardPullback} is actually defining sieves (one needs to check that they are downwards closed). This verification was formalized in \mathlib, and each point of Definition~\ref{def:functorPushforwardPullback} contains the corresponding external link.

\begin{definition}\label{def:functor_continuous_etc} In the same setting of Definition~\ref{def:functorPushforwardPullback}, denote by $\T$ the topology on $\C$ and by $\T'$ that on~$\D$.
  \begin{list_cases}
      \item We say that $F$ is \emph{continuous} if for every $\T'$-sheaf $P$ on $\D$, the presheaf $P \circ F^{\op}$ on $\C$ is a $\T$-sheaf \href{https://github.com/leanprover-community/mathlib4/blob/743032e7ead097fb3e8ae5cd02d29cdd8899161c/Mathlib/CategoryTheory/Sites/CoverPreserving.lean#L175-L178}{\faExternalLink}. In particular, if $F$ is continuous it induces a functor
      \[
          F^* \colon \Sh_{\T'}(\D) \longrightarrow \Sh_{\T}(\C).
      \]
      \item We say that $F$ is \emph{cocontinuous} if for every object $U$ of $\C$ and every $\T'$-covering sieve $S$ on $F(U)$, the functor-pullback $F^*S$ is a $\T$-covering sieve of $U$. \href{https://github.com/leanprover-community/mathlib4/blob/743032e7ead097fb3e8ae5cd02d29cdd8899161c/Mathlib/CategoryTheory/Sites/CoverLifting.lean#L69-L74}{\faExternalLink}
      \item We say that $F$ is \emph{cover-dense} if for every object $Y$ of $\D$, the $F$-image sieve $S^F_Y$ is a $\T'$-covering sieve.~\href{https://github.com/leanprover-community/mathlib4/blob/743032e7ead097fb3e8ae5cd02d29cdd8899161c/Mathlib/CategoryTheory/Sites/DenseSubsite.lean#L88-L94}{\faExternalLink}\label{pt:def_cont_etc:cover_dense}
  \end{list_cases}
\end{definition}

\begin{remark}
  Observe that in point~\ref{pt:def_cont_etc:cover_dense} of Definition~\ref{def:functor_continuous_etc} the topology $\T$ on $\C$ plays no role. Hence, to speak of cover-dense functors one only needs a Grothendieck topology on the target.
\end{remark}

\begin{proposition}\label{prop:sheafEquiv_of_cocontinuous_coverDense}
  \href{https://github.com/leanprover-community/mathlib4/blob/743032e7ead097fb3e8ae5cd02d29cdd8899161c/Mathlib/CategoryTheory/Sites/CoverLifting.lean#L319-L353}{\faExternalLink}
  \href{https://github.com/leanprover-community/mathlib4/blob/743032e7ead097fb3e8ae5cd02d29cdd8899161c/Mathlib/CategoryTheory/Sites/DenseSubsite.lean#L546-L551}{\faExternalLink}
  In the setting of Definition~\ref{def:functor_continuous_etc}, suppose that $F$ is continuous and cocontinuous. Then we have an adjunction $F^* \dashv F_*$. If $F$ is also fully faithful and cover-dense, then this adjunction is an adjoint equivalence of categories. 
\end{proposition}

\begin{definition}\label{def:inducedTopology}
  \href{https://github.com/leanprover-community/mathlib4/blob/743032e7ead097fb3e8ae5cd02d29cdd8899161c/Mathlib/CategoryTheory/Sites/InducedTopology.lean#L68-L95}{\faExternalLink}
  Let $\C$ and $\D$ be categories, let $\T'$ be a Grothendieck topology on $\D$ and let $F \colon \C \to \D$ be a fully faithful cover-dense functor. Define a Grothendieck topology $\T$ on $\C$ as follows: we declare that a sieve $S$ on an object $X$ in $\C$ is a $\T$-covering sieve if and only if the functor-pushforward sieve $F_*S$ is a $\T'$-covering sieve of $F(X)$ (see \href{https://github.com/leanprover-community/mathlib4/blob/743032e7ead097fb3e8ae5cd02d29cdd8899161c/Mathlib/CategoryTheory/Sites/InducedTopology.lean#L68-L95}{\faExternalLink} for a proof of the fact that this indeed defines a Grothendieck topology). This is called the \emph{topology induced by $F$}.
\end{definition}

\begin{lemma}\label{lem:coverLiftingPreserving_induced}
  \href{https://github.com/leanprover-community/mathlib4/blob/743032e7ead097fb3e8ae5cd02d29cdd8899161c/Mathlib/CategoryTheory/Sites/InducedTopology.lean#L152-L159}{\faExternalLink}
  Let $\C$ and $\D$ be categories, let $\T'$ a Grothendieck topology on $\D$ and let $F \colon \C \to \D$ be a fully faithful cover-dense functor. Equip $\C$ with the induced topology. Then $F$ is continuous and cocontinuous. 
\end{lemma}
\begin{proof}[Proof of Proposition~\ref{prop:sheafEquiv}]
 To show that $\C$ is precoherent, let $(\pi_i\colon X_i\to B)_i$ be a finite effective epimorphic family in $\C$ and let $f\colon B'\to B$ be any morphism. The family $F(\pi_i)$ is finite effective epimorphic (in $\D$) by condition~\ref{pt:propSheafEquiv:condReflects}: then, the hypothesis that $\D$ is precoherent, applied to the morphism $F(f)\colon F(B')\to F(B)$, provides a finite effective epimorphic family $\varpi_j\colon Y_j\to F(B')$ whose components factor through some of the $F(\pi_i)$. By condition~\ref{pt:propSheafEquiv:condExist} there exist objects $(X'_j)_j$ in $\C$ together with effective epimorphisms $\varphi_j\colon F(X_j')\to Y_j$, that combine into an effective epimorphic family $F(X_j')\to F(B')$ thanks to Lemma~\ref{lem:effective-epi-transitive}; moreover, the morphisms in this family are of the form $F(\pi_j')$ for suitable $\pi_j'\colon X_j'\to B'$ because $F$ is fully faithful. Applying again condition~\ref{pt:propSheafEquiv:condReflects}, this family $(\pi_j')_j$ is finite effective epimorphic; that, for each $j$, the morphism $\pi_j'$ factors through some of the $\pi_i$ follows from the equivalent statement for the components of $\varpi_j$, combined once more with the full faithfulness of $F$. This establishes point~\ref{pt:propSheafEquiv:precoherent}.

 We claim that the topology on $\C$ induced by $F$ is the coherent topology. It suffices to show that given an object $X$ of $\C$, a sieve $S$ on $X$ is a covering for the induced topology if and only if it contains a finite effective epimorphic family. Suppose first that $S$ contains a finite effective epimorphic family $(\pi_i \colon Y_i \to X)_i$. By condition~\ref{pt:propSheafEquiv:condReflects}, the family $(F(\pi_i))_i$ is finite effective epimorphic, and is clearly contained in $F_*S$. Hence $F_*S$ is a covering sieve of $F(X)$ with respect to the coherent topology on $\D$ by Proposition~\ref{prop:mem-coherent-topology-iff}, which means that $S$ is a covering sieve with for the induced topology (see Definition~\ref{def:inducedTopology}). For the other direction, suppose that $S$ is a covering sieve for the induced topology: as for the first implication, this is equivalent to the condition that $F_*S$ contains a finite effective epimorphic family $(\pi_i \colon Z_i \to F(X))_i$. Condition~\ref{pt:propSheafEquiv:condExist} ensures that for every $i$, there is an effective epimorphism of the form $f_i \colon F(Y_i) \to Z_i$; moreover, since $\C$ is precoherent, we can apply Lemma~\ref{lem:effective-epi-transitive} to obtain that the family $(\pi_i \circ f_i \colon F(Y_i) \to F(X))_i$ is effective epimorphic. Since $F$ is full and reflects finite effective epimorphic families by condition~\ref{pt:propSheafEquiv:condReflects}, this family can be pulled back to a finite effective epimorphic family $(Y_i \to X)_i$ contained in $F^\ast(F_\ast S)$. We conclude thanks to Proposition~\ref{prop:sheafEquiv_of_cocontinuous_coverDense}.  

 Endowing $\D$ with the coherent topology, point~\ref{pt:propSheafEquiv:comp} is now immediate from Lemma~\ref{lem:coverLiftingPreserving_induced} (see Definition~\ref{def:functor_continuous_etc}).
  
  To finish the proof, we pass to point~\ref{pt:propSheafEquiv:equiv}, again endowing $\D$ with the coherent topology. By Proposition~\ref{prop:sheafEquiv_of_cocontinuous_coverDense}, it suffices to prove that $F$ is cover-dense, continuous and cocontinuous. By Lemma~\ref{lem:coverLiftingPreserving_induced} and the above discussion, it suffices to prove that $F$ is cover-dense.
  By Proposition~\ref{prop:mem-coherent-topology-iff}, it suffices to show that for every object $Y$ of $\D$, the $F$-image sieve $S^F_Y$ contains an effective epimorphism. Condition~\ref{pt:propSheafEquiv:condExist} ensures the existence of an object $X$ in $\C$ and of an effective epimorphism $F(X) \to Y$, that, by definition of the $F$-image sieve, belongs to $S^F_Y$.  
\end{proof}

\begin{remark}
  A finite-coproduct preserving functor between finitary extensive categories preserves (\emph{resp.} reflects) finite effective epimorphic families if and only if it preserves (\emph{resp.}~reflects) effective epimorphisms. This is because finite effective epimorphic families in extensive categories are precisely those that induce effective epimorphisms on the coproduct (see Lemmas~\ref{lem:effectiveEpi_of_family} and~\ref{lem:effectiveEpiFamily_of_effectiveEpi}). This observation yields variants (see for instance~\href{https://github.com/leanprover-community/mathlib4/blob/743032e7ead097fb3e8ae5cd02d29cdd8899161c/Mathlib/CategoryTheory/Sites/Coherent/SheafComparison.lean#L131-L141}{\faExternalLink}) of Proposition~\ref{prop:sheafEquiv} 
  in the case where the target is preregular and finitary extensive and the functor preserves certain pullbacks and coproducts, or when the target category is already finitary extensive.
\end{remark}

\section{Condensed mathematics}\label{sec:back_to_condensed}
We can now introduce condensed sets and prove the main theorems from our general categorical results. We begin with the following result:
\begin{proposition} \label{prop:preregular_finitary_extensive}
The categories $\CompHaus$, $\Profinite$ and $\Stonean$ are preregular and finitary extensive.
\end{proposition}
\begin{proof}
    \href{https://github.com/leanprover-community/mathlib4/blob/743032e7ead097fb3e8ae5cd02d29cdd8899161c/Mathlib/Topology/Category/CompHaus/EffectiveEpi.lean#L87-L95}{\faExternalLink}
    \href{https://github.com/leanprover-community/mathlib4/blob/743032e7ead097fb3e8ae5cd02d29cdd8899161c/Mathlib/Topology/Category/Profinite/EffectiveEpi.lean#L104}{\faExternalLink}
    \href{https://github.com/leanprover-community/mathlib4/blob/743032e7ead097fb3e8ae5cd02d29cdd8899161c/Mathlib/Topology/Category/Stonean/EffectiveEpi.lean#L97}{\faExternalLink}
    \href{https://github.com/leanprover-community/mathlib4/blob/743032e7ead097fb3e8ae5cd02d29cdd8899161c/Mathlib/Topology/Category/CompHaus/Limits.lean#L237-L238}{\faExternalLink}
    \href{https://github.com/leanprover-community/mathlib4/blob/743032e7ead097fb3e8ae5cd02d29cdd8899161c/Mathlib/Topology/Category/Profinite/Limits.lean#L230-L231}{\faExternalLink}
    \href{https://github.com/leanprover-community/mathlib4/blob/743032e7ead097fb3e8ae5cd02d29cdd8899161c/Mathlib/Topology/Category/Stonean/Limits.lean#L321-L322}{\faExternalLink}
Let $\C$ denote any of the categories $\CompHaus$, $\Profinite$ or $\Stonean$. Note that the effective epimorphisms in $\C$ are precisely the continuous surjections (Proposition~\ref{prop:CompHausEffectiveEpi}). These also coincide with the epimorphisms, by Lemma~\ref{lemma:epi_iff_surj}. Given the explicit description of pullbacks in $\Profinite$ and $\CompHaus$, we immediately obtain that effective epimorphisms can be pulled back, and therefore $\Profinite$ and $\CompHaus$ are preregular. To see that $\Stonean$ is preregular, we use the fact that every object in $\Stonean$ is projective, and hence every epimorphism can be pulled back to the identity. 

We also need to show that $\C$ is finitary extensive. In \mathlib it was already proved that the category of all topological spaces is finitary extensive, and that given a functor $F \colon \C \to \D$ to a finitary extensive category which preserves and reflects finite coproducts, preserves pullbacks along coprojections in finite coproducts and reflects pullbacks, if $\C$ has finite coproducts and pullbacks along coprojections, then $\C$ is finitary extensive. To see that $\C$ together with its inclusion functor to the category of topological spaces has these properties, the only point which needs clarification is the case for pullbacks in $\Stonean$. We know that $\Stonean$ does not have all pullbacks, but in the very special case of coprojections in finite coproducts, it does. Indeed, these are clopen embeddings, in which case the pullback is identified with the preimage of the image of the embedding.
\end{proof}

\begin{definition}\label{def:condensed set}
    \href{https://github.com/leanprover-community/mathlib4/blob/743032e7ead097fb3e8ae5cd02d29cdd8899161c/Mathlib/Condensed/Basic.lean#L38-L43}{\faExternalLink}
    A \emph{condensed set} is a sheaf for the coherent topology on $\CompHaus$. (We refer the reader to Definition~\ref{def:coherent coverage} for the definition of coherent topology.)
\end{definition}

\begin{remark}
Thanks to Theorem~\ref{thm:condensed set equivalence} below, a condensed set can be defined as a sheaf for the coherent topology on $\Profinite$ or $\Stonean$.
\end{remark}
\begin{theorem} \label{thm:sheaves}
\href{https://github.com/leanprover-community/mathlib4/blob/743032e7ead097fb3e8ae5cd02d29cdd8899161c/Mathlib/CategoryTheory/Sites/Coherent/SheafComparison.lean#L235-L242}{\faExternalLink}
\href{https://github.com/leanprover-community/mathlib4/blob/743032e7ead097fb3e8ae5cd02d29cdd8899161c/Mathlib/CategoryTheory/Sites/Coherent/SheafComparison.lean#L249-L253}{\faExternalLink}
    \begin{list_cases}
    \item  When $\C$ is $\CompHaus$ or $\Profinite$, a presheaf $X \colon \C^{\op} \to \Set$ is a sheaf for the coherent topology on~$\C$ if and only if it satisfies the following two conditions:\label{main_thm:pt_comphaus}
    \begin{list_conditions}
        \item $X$ preserves finite products: in other words, for every finite family $(T_i)$ of object of $\C$, the natural map 
        \[
            X\Bigl(\coprod_i T_i\Bigr) \longrightarrow \prod_i X(T_i)    
        \]
        is a bijection.\label{main_thm:pt_comphaus:cond_bij}
        \item For every surjection $\pi \colon S \to T$ in $\C$, the diagram 
        \[\begin{tikzcd}
            {X(T)} & {X(S)} & {X(S \times_T S)}
            \arrow["{X(\pi)}", from=1-1, to=1-2]
            \arrow[shift left, from=1-2, to=1-3]
            \arrow[shift right, from=1-2, to=1-3]
        \end{tikzcd}\]
        is an equalizer \textup{(}the two parallel morphisms being given by the projections in the pullback\textup{)}. \label{main_thm:pt_comphaus:cond_equalizer}
    \end{list_conditions}
    \item A presheaf $X \colon \Stonean^{\op} \to \Set$ is a sheaf for the coherent topology on $\Stonean$ if and only if it preserves finite products: in other words, for every finite family $(T_i)$ of object of $\C$, the natural map 
    \[
        X\Bigl(\coprod_i T_i\Bigr) \longrightarrow \prod_i X(T_i)    
    \]
    is a bijection.\label{main_thm:pt_stonean}
\end{list_cases}
\end{theorem}
\begin{proof}
In the case when $\C$ is $\CompHaus$ or $\Profinite$, it has all pullbacks and we obtain the characterization from Proposition~\ref{prop:regular_extensive_sheaf}. In the case of $\Stonean$, since every object is projective, we obtain the characterization from Proposition~\ref{prop:regular_extensive_sheaf_projective}. 
\end{proof}
\begin{remark}
    A detailed proof of Theorem~\ref{thm:sheaves} is given in~\cite[Theorems~1.2.17 and~1.2.18]{msc-dagur}.
\end{remark}
\begin{remark}
    A condition similar to the one in point~\ref{main_thm:pt_comphaus} of Theorem~\ref{thm:sheaves} above holds true when $\C$ is $\Stonean$ as well, except that condition~\ref{main_thm:pt_comphaus}-\ref{main_thm:pt_comphaus:cond_bij} must be modified slightly (for example, using $1$-hypercovers) due to the fact that $\Stonean$ does not have pullbacks. The content of~\ref{main_thm:pt_stonean} is that this analogous condition turns out to be vacuously true in $\Stonean$.
\end{remark}
\begin{theorem} \label{thm:condensed set equivalence}
    \href{https://github.com/leanprover-community/mathlib4/blob/743032e7ead097fb3e8ae5cd02d29cdd8899161c/Mathlib/Condensed/Equivalence.lean#L86-L92}{\faExternalLink}
    \href{https://github.com/leanprover-community/mathlib4/blob/743032e7ead097fb3e8ae5cd02d29cdd8899161c/Mathlib/Condensed/Equivalence.lean#L43-L49}
    {\faExternalLink}
    The inclusion functors $\Profinite \to \CompHaus$ and $\Stonean \to \CompHaus$ induce equivalences of categories between the categories of sheaves for the coherent topology on $\CompHaus$, $\Profinite$, and $\Stonean$.
\end{theorem}
\begin{proof}
    We are going to apply Proposition~\ref{prop:sheafEquiv}. We spell out the case of $\Stonean$ here, the other one being similar. It is clear that the inclusion functor preserves and reflects effective epimorphisms (by the characterization of effective epimorphisms as continuous surjections). Verifying the other condition in the theorem amounts to proving that $\CompHaus$ has enough projectives. Given a compact Hausdorff space $S$, denote by $S^\delta$ the set $S$ equipped with the discrete topology. Then the Stone--\Cech compactification $\beta S^\delta$ is a projective object with a continuous surjection $\beta S^\delta \to S$.
\end{proof} 

\section*{Acknowledgments}
This work began when the five authors gathered in Copenhagen for the \emph{Masterclass: Formalisation of Mathematics} that took place in June 2023. We are grateful to Kevin Buzzard for the lectures he delivered jointly with one of the authors (A.~T.) and to Boris Kjær who co-organized the masterclass with another author (D.~A.). Financial support for the masterclass was provided by the Copenhagen Centre for Geometry and Topology (GeoTop) and D.~A.~ acknowledges funding from GeoTop through grant CPH-GEOTOP-DNRF151. N.~K.~was supported by the Research Council of Norway grant 302277 -- \emph{Orthogonal gauge duality and non-commutative geometry}. F.~N.~ was supported by a \emph{projet émergent} from Labex MILYON/ANR-10-LABX-0070. A.~T.~is funded by NSERC discovery grant RGPIN-2019-04762.

All authors wish to express their gratitude to the whole \mathlib community for their support and interest in this work, and in particular to Johan Commelin and to Joël Riou for reviewing and improving many pull requests related to the work presented here.

\bibliographystyle{plainnat}
\bibliography{references}

\end{document}